\documentclass[reqno]{amsart}
\usepackage{fullpage}
\usepackage[english]{babel}
\usepackage{amsmath}
\usepackage{amssymb}
\usepackage{selinput}
\SelectInputMappings{eacute={é}}
\usepackage{verbatim}
\usepackage{bbm}
\usepackage{enumerate}
\usepackage{latexsym}
\usepackage{color}
\usepackage{amsfonts}
\usepackage{mathrsfs}
\usepackage{esint}

\theoremstyle{plain}
\newtheorem{prop}{Proposition}[section]
\newtheorem{thm}{Theorem}[section]
\newtheorem{lem}{Lemma}[section]
\newtheorem{rem}{Remark}[section]
\newtheorem{definition}{Definition}[section]

\DeclareMathOperator{\TV}{TV}
\newcommand{\Lip}{\mathscr{L}}
\newcommand{\R}{\mathbb{R}}
\DeclareMathOperator{\sign}{sign}
\DeclareMathOperator{\supp}{supp}

\title[Follow-the-leader scheme with space dependent flux]{Convergence of the follow-the-leader scheme for scalar conservation laws with space dependent flux}

\author{M. Di Francesco and G. Stivaletta}

\address{Department of Information Engineering, Computer Science, and Mathematics, Via Vetoio 1, Coppito, I-67100
L’Aquila, Italy.}

\email{marco.difrancesco@univaq.it, graziano.stivaletta@graduate.univaq.it}

\date{}

\begin{document}

\begin{abstract}
This paper deals with the derivation of entropy solutions to Cauchy problems for a class of scalar conservation laws with space-density depending fluxes from systems of deterministic particles of follow-the-leader type. We consider fluxes which are product of a function of the density $v(\rho)$ and a function of the space variable $\phi(x)$. We cover four distinct cases in terms of the sign of $\phi$, including cases in which the latter is not constant. The convergence result relies on a local maximum principle and on a uniform $BV$ estimate for the approximating density.
\end{abstract}

\maketitle

\section{Introduction}

\subsection{Continuum vs discrete modelling: the example of traffic flow modelling}

The approximation of nonlinear transport equations via follow-the-leader type schemes has attracted a lot of attention in the recent years. As a paradigm, consider Lighthill-Whitham-Richards' (LWR) equation for traffic flow \cite{LW,Richards}
\begin{equation}\label{eq:LWR}
  \rho_t + (\rho v(\rho))_x = 0,
\end{equation}
where $\rho$ is the density of vehicles and $\rho\mapsto v(\rho)$ is a decreasing function modelling the Eulerian velocity of vehicles. As it is well known, in this model instantaneous response to the distance to the preceding vehicle is assumed by neglecting drivers' reaction time, whereas other models \cite{ARZ} take the latter into account. However, \eqref{eq:LWR} is considered as a reliable model in several situations, for instance with low densities, see e.g. the recent book \cite{rosini} and the references therein. 

Both approaches in \cite{LW,Richards} and \cite{ARZ} treat the density of cars as a \emph{continuum}, that is as a medium that can be divided into particles of arbitrary small mass without changing the physical nature of the system. The continuum approach has many advantages, particularly in that it provides explicitly computable solutions combining nonlinear shock-waves and rarefaction waves in relevant examples and it allows extensions to control problems and modelling on networks in a relatively simple way. However, a \emph{discrete approach} is more suitable to model single drivers' behavior (car following, free driving) and is more convenient to simulate multi lane flow with lane changing. Neglecting the driver's time reaction as the distance to the preceding vehicle changes, the simplest discrete law for the dynamics of $n+1$ drivers is provided by the \emph{follow-the-leader} system
\begin{equation}\label{eq:FTLintro}
  \begin{cases}
    \displaystyle\dot{x}_i(t)=v\biggl(\frac{\ell}{x_{i+1}(t)-x_i(t)}\biggr), \quad \text{for}\ i \in \{ 0,\ldots,n-1 \},
    \\
    \dot{x}_n(t)=v_{\max}=v(0),
  \end{cases}
\end{equation}
where $x_0(t)<\ldots<x_n(t)$ denote the positions of the $n+1$ vehicles at time $t$, $v$ is a given non-negative, non-increasing function on $[0,+\infty)$ with finite value $v_{\max}$ at $0$ and $\ell$ is the (one-dimensional) mass of each vehicle. Typically, a maximum density $\rho_{\max}$ is prescribed in the model in order to avoid collisions, and the velocity $v$ satisfies $v(\rho_{\max}) = 0$. The vehicle $x_n$, called `leader', travels with maximum speed as no vehicles are ahead of it. The finite dimensional dynamical system \eqref{eq:FTLintro} is usually coupled with $n+1$ initial conditions $x_i(0)=\bar{x}_i$, $i=0,\ldots,n$.

\subsection{The follow-the-leader approximation of LWR equation}

The ODE system \eqref{eq:FTLintro} and the PDE \eqref{eq:LWR} are strictly related. To better understand such a statement, we observe that the quantity $\frac{\ell}{x_{i+1}(t)-x_i(t)}$ in \eqref{eq:FTLintro} has the physical dimension of a one-dimensional density. Therefore, loosely speaking, a reasonable continuum version of the ODE in \eqref{eq:FTLintro} is
\begin{equation}\label{eq:loosely}
  \dot{x}(t)= v(\rho(t)),
\end{equation}
where $x(t)$ denotes a Lagrangian trajectory of a particle with infinitesimal mass and $\rho(t)$ the density computed in the infinitesimal region around the same particle. Now, \eqref{eq:loosely} is nothing but a Lagrangian formulation of \eqref{eq:LWR}. Indeed, if we look at the latter as a continuity equation, the term $v(\rho(x,t))$ in it denotes the \emph{kinetic velocity} of a particle located at $x$ at time $t$.

Hence, the equation in \eqref{eq:FTLintro} for $i<n$ can be considered as a \emph{Lagrangian discrete counterpart} of the continuity equation \eqref{eq:LWR}. This motivates the mathematical interest for \eqref{eq:FTLintro} as a possible \emph{many-particle approximation} of \eqref{eq:LWR}, that is in the limit as $n\rightarrow +\infty$. While this fact has been largely known in the literature as a `formal limit', the result in \cite{DFR} proved it as a rigorous mathematical result. More precisely, the result in \cite{DFR} can be stated as follows: take an arbitrary continuum initial condition $\overline{\rho}\in L^1(\mathbb{R})\cap L^\infty(\mathbb{R})$ with compact support, and consider a suitable \emph{atomization} of $\overline{\rho}$, for instance a set of $n+1$ particles $\bar{x}_0,\ldots,\bar{x}_n\in \mathbb{R}$ with the property that $\int_{\bar{x}_i}^{\bar{x}_{i+1}}\overline{\rho}(x) dx = 1/n$ for all $i=0,\ldots,n-1$. Now, consider the (unique) solution to \eqref{eq:FTLintro} with initial condition $\bar{x}_i$, $i=0,\ldots,n$ and the discrete piecewise reconstruction of the particles' density
\begin{equation}\label{eq:rhoN_intro}
  \rho^n(x,t)=\sum_{i=0}^{n-1} \frac{1}{n(x_{i+1}(t)-x_i(t))}\mathbbm{1}_{[x_i(t),x_{i+1}(t))}(x).
\end{equation}
Then, $\rho^n$ converges in $L^1_{loc}(\mathbb{R}\times \mathbb{R}_{+})$ as $n\rightarrow +\infty$ to the unique entropy solution (in the sense of Kru{\v{z}}kov \cite{kruzkov}) $\rho$ of the scalar conservation law \eqref{eq:LWR} with initial datum $\overline{\rho}$. Such result was later extended to a larger class of initial conditions in \cite{DFRF}. An alternative proof was provided later on in \cite{HolRisebro2}, see also the numerical result in \cite{HolRisebro}. For the concept of entropy solution we refer e.g. to \cite{dafermos_book} at this stage (see also \cite{rosini} for the particular case of traffic flow modelling). We shall recall this concept at the beginning of Section 4 based on the definition provided in \cite{KarlsenRisebro}.

\subsection{A literature review on many particle limits for transport equations}

In a more general framework in which the dependence on $\rho$ in the velocity term $v$ in \eqref{eq:LWR} includes possible diffusion terms, or external force fields, or nonlocal interaction terms, several results are available in the literature. We provide here a partial list of results. A probabilistic approach based on exclusion processes was developed in several works, we mention here \cite{ferrari1,ferrari2,liggett}. When diffusion terms are included, we mention here the milestone results in \cite{vara1,vara2}. System \eqref{eq:FTLintro} is a typical example of \emph{deterministic} particle system, in that no stochastic effects are considered and the position of each particle is exactly computable for all times $t\geq 0$. A first attempt to detect diffusion effect via deterministic particles is due to \cite{russo}. The result in \cite{gosse} extends this approach to nonlinear diffusions. A relevant recent result also involving external potentials is contained in \cite{matthes}. Deterministic particle limits are also relevant in the literature of the modelling of swarming phenomena, see e.g. \cite{choi} and the references therein.

The result in \cite{DFR} provided, for the first time, an abstract validation of the continuum model \eqref{eq:LWR} as a `good' approximation of the discrete model \eqref{eq:FTLintro}. In particular, the emergence of typical patterns such as rarefaction waves and shocks - that are easily computable in \eqref{eq:LWR} - is established as a phenomenon that occurs also in the discrete setting, in a proper scaling regime in which the domain is large enough to include a very large number of vehicles and the total mass of the vehicles is normalized. We stress that said patterns are not detectable analytically in \eqref{eq:FTLintro} for finite $n$. Moreover, this set of results is relevant also from the numerical point of view, as it allows to follow the movement of each vehicle unlike standard approaches such as classical Godunov type methods. Finally, these results hold without prescribing the initial condition to be far from the vacuum state (a restriction which would contradict the fact that the inertia-free approach of \eqref{eq:LWR} is more suitable for low densities). We mention at this stage that the literature contains several results about the derivation of the second order ARZ model via deterministic follow-the-leader systems, see e.g. \cite{Aw,degond}.

Although traffic flow is a motivating example to justify \eqref{eq:LWR} as a many-particle limit for \eqref{eq:FTLintro}, the results in \cite{DFR,DFRF} hold under more general assumptions on the velocity map $v$: it suffices to assume that $v$ is monotone (decreasing or increasing) and the monotonicity of $v$ determines the proper upwind direction for the discrete density on the right-hand side of \eqref{eq:FTLintro} (an increasing velocity $v$ requires the use of the backward density $\frac{\ell}{x_{i}(t)-x_{i-1}(t)}$). Indeed, this deterministic particle approach to solving nonlinear continuity equations was later on extended to other models. In \cite{DFFRR2} the same approach was used to approximate solutions to \eqref{eq:LWR} on a bounded domain with Dirichlet type conditions. In \cite{DFFRR1} a suitable modification of \eqref{eq:FTLintro} was proven to converge in the many-particle limit to weak solutions to the Hughes model for pedestrian movements in one space dimension. A nonlocal version of \eqref{eq:LWR} was considered in \cite{DFFRad} as the many-particle limit of a suitable variant of \eqref{eq:FTLintro} considering nonlocal interactions with all particles, see also \cite{FRad} for the same model coupled with a nonlinear (degenerate) diffusion. We also mention here the results related to the ARZ model in \cite{DFFR_ARZ,bergoa}.

\subsection{Our result: the case with space-dependent flux}

The present paper contributes to the line of research sketched above by considering the case of a scalar conservation law with space-dependent flux
\begin{equation}\label{eq:main_intro}
  \rho_{t} + \bigl(\rho v(\rho) \phi(x)\bigr)_{x} =0,
\end{equation}
where $v$ is monotone and $\phi$ is a given external drift term depending on the position $x$. Besides being well motivated in the modelling context of traffic flow - for example in situations in which the speed of the vehicles is also affected by external factors (such as temporary road maintenance, or sudden turns or rises) - the equation \eqref{eq:main_intro} has a pretty wide range of potential applications in sedimentation processes \cite{buerger_et_al}, flow of glaciers \cite{glaciers}, formation of Bose-Einstein condensates \cite{toscani_kaniadakis}. For a more general description of the applications of nonlinear scalar conservation laws we refer to \cite{dafermos_book} and the references therein.

Similarly to the approach of \cite{DFR} and later results, we will assume throughout this paper that $v:[0,+\infty)\rightarrow [0,+\infty)$ is \emph{monotone non-increasing} and non-negative, with $v(0)<+\infty$. A symmetric result could be stated in case of a non-decreasing $v$, we shall omit the details. As for the potential $\phi$, we consider four cases:
\begin{itemize}
\item[(P1)] $\phi(x)\geq 0$ for all $x\in \mathbb{R}$ (forward movement);
\item[(P2)] $\phi(x)\leq 0$ for all $x\in \mathbb{R}$ (backward movement);
\item[(P3)] $x\phi(x)\geq 0$ for all $x\in \mathbb{R}$ (repulsive movement);
\item[(P4)] $x \phi(x)\leq 0$ for all $x\in \mathbb{R}$ (attractive movement).
\end{itemize}

We refer to section \ref{sec:statement} for the precise statement of all assumptions on $v$ and $\phi$. For each of the above four cases we shall provide an ad-hoc many-particle approximation result in the spirit of \eqref{eq:FTLintro}. For example, case (P1) requires the use of the \emph{forward} follow-the-leader scheme
\begin{equation}\label{eq:FTL_intro_paper}
  \begin{cases}
    \displaystyle\dot{x}_i(t)=v\left(\frac{\ell}{x_{i+1}(t)-x_i(t)}\right) \phi(x_i(t)), & \mbox{for $i=0,\ldots,n-1,$}  \\
    \dot{x}_n(t)=v(0)\phi(x_n(t)). &
  \end{cases}
\end{equation}

The distinction between case (P1) and case (P2) is relevant in that it implies a change in the upwind direction of the scheme. More in detail, if $\phi\leq 0$ then all particles are subject to a drift directed towards the negative direction. Hence, it is reasonable to assume that each particle adjusts its speed by considering the distance to its \emph{left} nearest neighbor, and the leftmost particle will be the leader travelling with $v=v(0)$. In case (P3), the drift direction changes at the origin $x=0$, with a positive direction on $x\geq 0$ (non-negative $\phi$) and a negative one on $x\leq 0$ (non-positive $\phi$). We shall refer to this case as \emph{repulsive} movement, since it implies a drift of all particles away from the origin. Two leaders (leftmost and rightmost particle) will travel with speed $v=v(0)$. Symmetrically, in case (P4) particles move towards the positive (negative respectively) direction on $x\leq 0$ (on $x\geq 0$ respectively). This implies that no actual `leader' exists in the sense of the previous cases, and particles adapt their speed with respect to the relative position with their right (left respectively) nearest neighbor on $x\leq 0$ (on $x\geq 0$ respectively). This situation implies an \emph{attractive} movement towards the origin, a phenomenon that could potentially imply collision between the two particles nearest to the origin in a finite time. To see this, consider the example $v(\rho)=(1+\rho)^{-1}$, $\phi(x)=-|x|^\alpha$ with $\alpha\in (0,1)$. Setting two particles at initial positions $-x_0, x_0$ with $x_0>0$, one can easily show that the two particles $-x(t)$ and $x(t)$ obeying
\[\dot{x}(t)=\phi(x(t))v\left(\frac{\ell}{2x(t)}\right)\,,\qquad x(0)=x_0,\]
reach the origin in a finite time. Other significant examples originate in the continuum setting in the study of Bose-Einstein condensates, see \cite{CDFT_ANS}, with $\phi(x)=-x$ and $v(\rho)=\rho^2$, in which the finite time blow-up in $L^\infty$ of the density is proven. In order to bypass this problem, we shall require an additional assumption for case (P4), namely that the velocity map $v(\rho)$ vanishes at some prescribed maximal density value $R_{\max}$ and is equal to zero on $[R_{\max},+\infty)$. Such assumption is reasonable in contexts such as traffic flow in a single lane, in which overtaking of vehicles is not allowed. 

Our choice in cases (P3) and (P4) to consider just one point at which the potential $\phi$ changes its sign is only motivated by the sake of simplicity. By suitably combining the results obtained in the previous cases, one can easily extend our result to a potential $\phi$ satisfying the following assumption:
\begin{itemize}
  \item $\phi$ has finitely many zeroes $\lambda_1<\ldots,\lambda_k$, it has constant sign on each interval $(\lambda_j,\lambda_{j+1})$ for all $j=1,\ldots,k-1$ and it changes its sign near each $\lambda_j$.
\end{itemize}
More precisely, we can split the real line into separate regions with attractive potentials and apply the strategy we propose for case (P4). Clearly, the assumption of a velocity function $v$ vanishing at a fixed maximal density should hold in order to avoid blow-up. We omit the details.

In all the aforementioned four cases we are able to prove a convergence result in the spirit of \cite{DFR}: given an initial condition $\bar{\rho}\in L^\infty(\mathbb{R})\cap BV(\mathbb{R})$ non-negative and with compact support, we atomize $\bar{\rho}$ by a set of $n+1$ particles $\bar{x}_0,\ldots,\bar{x}_n$, we consider the piecewise constant density $\rho^n$ as in \eqref{eq:rhoN_intro} with $x_0(t),\ldots,x_n(t)$ solution to a suitable follow-the-leader scheme (\eqref{eq:FTL_intro_paper} in case (P1) as an example) with initial datum $\bar{x}_0,\ldots,\bar{x}_n$, and prove that $\rho^n$ converges locally in $L^1_{x,t}$ towards the unique entropy solution to \eqref{eq:main_intro} with $\bar{\rho}$ as initial condition. Such result requires as crucial steps:
\begin{itemize}
  \item A local maximum principle showing that $\|\rho^n\|_{L^\infty(\mathbb{R})}$ is uniformly bounded with respect to $n$ on arbitrary time intervals $[0,T]$;
  \item $BV$ compactness estimates;
  \item Consistency with the definition of entropy solutions (in the Kru{\v{z}}kov's sense \cite{kruzkov}) in the $n\rightarrow +\infty$ limit.
\end{itemize}
Such a strategy requires an $L^1\cap L^\infty$ setting. This is why we cannot consider case (P4) in presence of blow-up or concentration phenomena. This issue will be tackled in a future paper.

The paper is structured as follows. In section \ref{sec:statement} we define our four approximation schemes and prove their main properties, including the maximum principle for all of them. We highlight that cases (P1)-(P2)-(P3) feature a maximum principle in terms of the initial $L^\infty$ norm, whereas in case (P4) the uniform bound for $\rho^n$ is provided in terms of the maximal density $R_{\max}$. In section \ref{sec:compactness} we prove the needed uniform $BV$ estimate, as well as an equicontinuity property with respect to the Wasserstein distance that provide local $L^1$ compactness in space and time. Finally, in section \ref{sec:consistency} we state and prove our main result in Theorem \ref{thm:main}, that collects the convergence of the scheme in all four cases.

\section{Statement of the problem and maximum principles}\label{sec:statement}
Let us consider the following Cauchy problem for a one-dimensional conservation law
\begin{equation}
\label{sis2}
    \begin{cases}
     \rho_{t} + \bigl(\rho v(\rho) \phi(x)\bigr)_{x} =0,
    \qquad  & x\in \mathbb{R},\, t>0,
    \\ \rho(x,0)=\overline{\rho}(x),
    \qquad  & x\in \mathbb{R},
   \end{cases}
\end{equation}
where we assume that the function $v$ and the initial datum $\overline{\rho}$ satisfy respectively
\begin{itemize}
\item[(V)] $v\in\ Lip(\mathbb{R_{+}})$ is a non-negative and non-increasing Lipschitz function with $v(0):= v_{\max}< +\infty$;
\item[(I)] $\overline{\rho}\in \ L^{\infty}(\mathbb{R})\cap BV(\mathbb{R})$ is a non-negative, compactly supported function.
\end{itemize}
Concerning the potential $\phi$, we shall deal with four different cases
\begin{itemize}
\item[(P1)] $\phi(x)\geq 0$ for all $x\in \mathbb{R}$ (forward movement);
\item[(P2)] $\phi(x)\leq 0$ for all $x\in \mathbb{R}$ (backward movement);
\item[(P3)] $x\phi(x)\geq 0$ for all $x\in \mathbb{R}$ (repulsive movement);
\item[(P4)] $x \phi(x)\leq 0$ for all $x\in \mathbb{R}$ (attractive movement).
\end{itemize}
In all these cases, we assume the basic condition
\begin{itemize}
\item[(P)] $\phi\in W^{2,\infty}(\mathbb{R})$
\end{itemize}
and in the last one we add the following condition on the function $v$
\begin{itemize}
\item[(V$^*$)] There exists $R_{\max}> 0$ such that $\overline{R}:=||\overline{\rho}||_{L^{\infty}(\mathbb{R})}\leq R_{\max}$, $v(\rho)> 0$ for $\rho< R_{\max}$ and $v(\rho)\equiv 0$ for $\rho\geq R_{\max}$.
\end{itemize}
For the sake of simplicity, we suppose that the initial mass is normalised, that is
\[||\overline{\rho}||_{L^{1}(\mathbb{R})}=1.\] Moreover, let us denote with
\[[\overline{x}_{\min},\overline{x}_{\max}] =\mathrm{Conv}( \mathrm{supp}(\overline{\rho}))\]
the convex hull of the support of $\overline{\rho}$.

Our next goal is to provide an initial condition for the follow-the-leader systems. To perform this task, we split the interval $[\overline{x}_{\min},\overline{x}_{\max}]$ into $n$ sub-intervals having equal mass $\ell_{n}:=1/n$. So, for a fixed $n\in \mathbb{N}$ sufficiently large, we set $\overline{x}_{0}^{n}:=\overline{x}_{\min}$, $\overline{x}_{n}^{n}:=\overline{x}_{\max}$ and we define recursively
\begin{equation*}
\overline{x}_{i}^{n}:= \sup \biggl\{ x\in \mathbb{R}: \int_{\overline{x}_{i-1}^{n}}^{x}\overline{\rho}(x)dx < \ell_{n}\biggr\} \quad \text{for}\ i\in \{1, \dots, n-1\}.
\end{equation*}
From the previous definition we immediately have that $\overline{x}_{0}^{n}<\overline{x}_{1}^{n}<\dots <\overline{x}_{n}^{n}$ and
\begin{equation}
\label{5}
\int_{\overline{x}_{i-1}^{n}}^{\overline{x}_{i}^{n}}\overline{\rho}(x)dx= \ell_{n} \quad \text{for}\ i\in \{1,\dots,n-1\}.
\end{equation}

Next we introduce the follow-the-leader systems describing the evolution of the $n+1$ particles with initial positions $\overline{x}_{i}^{n}$, $i=0,\ldots,n$. The definition of the particle system depends on the cases (P1)-(P4) introduced above, hence we should introduce four different approximation schemes, nevertheless, as we will see in a moment, the latter two are strictly related to the former two. Cases (P1) and (P2) are the simplest ones, as the constant sign of $\phi$ does not affect the monotonicity of the velocity field $v(\rho)\phi(x)$. Consistently with the homogeneous case \cite{DFR}, when $\phi$ is non-negative we have that the velocity field decreases with respect to the density $\rho$. Therefore, in case (P1) the follow-the-leader scheme should consider a \emph{forward} finite-difference approximation of the density. Symmetrically, (P2) implies a \emph{backward} approximation. As a consequence, with the notation
\begin{equation*}
R_{i}^{n}(t):=\frac{\ell_{n}}{x_{i+1}^{n}(t)-x_{i}^{n}(t)},\qquad t\geq 0,\quad i\in \{0,\dots,n-1\},
\end{equation*}
in case (P1) we use the ODE system
\begin{equation}
\label{1}
\begin{cases}
\dot{x}_{i}^{n}(t)=v(R_{i}^{n}(t))\phi(x_{i}^{n}(t)), & \quad \text{for} \ i\in \{0,\dots,n-1\},
\\
\dot{x}_{n}^{n}(t)=v_{\max}\phi(x_{n}^{n}(t)),
\\
x_{i}^{n}(0)=\overline{x}_{i}^{n}, & \quad \text{for} \ i\in \{0,\dots,n\}
\end{cases}
\end{equation}
and in case (P2) we use
\begin{equation}
\label{2}
\begin{cases}
\dot{x}_{0}^{n}(t)=v_{\max}\phi(x_{0}^{n}(t)),
\\
\dot{x}_{i}^{n}(t)=v(R_{i-1}^{n}(t))\phi(x_{i}^{n}(t)), & \quad \text{for} \ i\in \{1,\dots,n\},
\\
x_{i}^{n}(0)=\overline{x}_{i}^{n}, & \quad \text{for} \ i\in \{0,\dots,n\}.
\end{cases}
\end{equation}
For cases (P3) and (P4) we consider a sort of combination of the previous two cases. With the notation
\begin{equation*}
k_{n}:=\max\bigl\{i\in \{0,\dots,n\}: \overline{x}_{i}^{n}\leq 0\bigr\},
\end{equation*}
the ODE system in case (P3) is
\begin{equation}
\label{3}
\begin{cases}
\dot{x}_{0}^{n}(t)=v_{\max}\phi(x_{0}^{n}(t)),
\\
\dot{x}_{i}^{n}(t)=v(R_{i-1}^{n}(t))\phi(x_{i}^{n}(t)), & \quad \text{for} \ i\in \{1,\dots,k_{n}\},
\\
\dot{x}_{i}^{n}(t)=v(R_{i}^{n}(t))\phi(x_{i}^{n}(t)), & \quad \text{for} \ i\in \{k_{n}+1,\dots,n-1\},
\\
\dot{x}_{n}^{n}(t)=v_{\max}\phi(x_{n}^{n}(t)),
\\
x_{i}^{n}(0)=\overline{x}_{i}^{n}, & \quad \text{for} \ i\in \{0,\dots,n\},
\end{cases}
\end{equation}
whereas in case (P4) we use
\begin{equation}
\label{4}
\begin{cases}
\dot{x}_{i}^{n}(t)=v(R_{i}^{n}(t))\phi(x_{i}^{n}(t)), & \quad \text{for} \ i\in \{0,\dots,k_{n}\},
\\
\dot{x}_{i}^{n}(t)=v(R_{i-1}^{n}(t))\phi(x_{i}^{n}(t)), & \quad \text{for} \ i\in \{k_{n}+1,\dots,n\},
\\
x_{i}^{n}(0)=\overline{x}_{i}^{n}, & \quad \text{for} \ i\in \{0,\dots,n\}.
\end{cases}
\end{equation}
For the sequel, we denote with $\Lip_{\psi}$ the Lipschitz constant of any $\psi\in Lip$ and we define the quantities
\begin{equation*}
\label{150}
L:= v_{\max}||\phi||_{L^{\infty}(\mathbb{R})}\quad \text{and} \quad L':= v_{\max}||\phi'||_{L^{\infty}(\mathbb{R})},
\end{equation*}
which are, thanks to assumption (P), two positive constant and, moreover, we drop the $n$ and $t$ dependence for simplicity, whenever there is no ambiguity.

We remark that the Lipschitz conditions on $v$ and $\phi$ ensure the local existence and uniqueness of solution to \eqref{1}-\eqref{4}. In particular, as long as the solution exists particles maintain the same order (otherwise one of the ratios $R_i$ blows up and the right-hand-side in \eqref{1}-\eqref{4} becomes meaningless). In order to safeguard global existence, we need to prove three properties:
\begin{itemize}
\item[(a)] Particles have a finite position and velocity on bounded time intervals,
\item [(b)] Particles always move in the same direction in \eqref{3} and \eqref{4},
\item [(c)] Particles never collide, consequently they always maintain the same order.
\end{itemize}

\begin{lem}[Finite position and velocity on bounded time intervals]
Assume (V), (I) and (P) are satisfied and, moreover, assume (V$^*$) is satisfied in case (P4). Then, as long as the solution to one of \eqref{1}-\eqref{4} exists, we have that
\begin{equation}
\label{190}
|x_{i}^{n}(t)|<+\infty \quad \text{and}  \quad |\dot{x}_{i}^{n}(t)|\leq L< +\infty \quad \text{for all} \ i\in \{0,\dots,n\}.
\end{equation}
\end{lem}
\begin{proof}
Due to assumption (P) and (V) we immediately get $|\dot{x}_{i}(t)|\leq L$. Furthermore, integrating the ODE defining the evolution of each $x_{i}(t)$, it follows that
\begin{equation*}
|x_{i}(t)|=\biggl|\overline{x}_{i}+\int_{0}^{t}\dot{x}_{i}(s)ds\biggr|\leq |\overline{x}_{i}|+\int_{0}^{t}|\dot{x}_{i}(s)|ds\leq |\overline{x}_{i}|+Lt<+\infty.
\end{equation*}
\end{proof}

\begin{lem}[Upper bound for the distance of two consecutive particles]
\label{lem7} Assume (V), (I) and (P) are satisfied and, moreover, assume (V$^*$) is satisfied in case (P4). Then, as long as the solution to one of \eqref{1}-\eqref{4} exists, there exists a positive constant $c$, depending only on $\phi$ and $v$, such that
\begin{equation}
\label{130}
x_{i+1}^{n}(t)-x_{i}^{n}(t)\leq x_{n}^{n}(t)-x_{0}^{n}(t)\leq \overline{x}_{\max}-\overline{x}_{\min}+ct \quad \text{for all} \  i\in \{0,\dots,n-1\}.
\end{equation}
Moreover, $c=0$ in case (P4).
\end{lem}
\begin{proof}
Integrating the ODE defining the evolution of the difference $x_{n}(t)-x_{0}(t)$, since all particles maintain the same order as long as the solution exists, we immediately get
\begin{equation*}
x_{i+1}(t)-x_i(t)\leq x_{n}(t)-x_{0}(t)\!=\! \biggl|\overline{x}_{n}-\overline{x}_{0}+\!\int_{0}^{t}\!\!\bigl(\dot{x}_{n}(s)-\dot{x}_{0}(s)\bigr)ds\biggr|\!\leq \overline{x}_{\max}-\overline{x}_{\min}+2Lt.
\end{equation*}
In case (P4) the estimate with $c=0$ follows from the fact that $x_n$ has a negative velocity and $x_0$ has a positive one.
\end{proof}

Now we prove the following technical result regarding cases (P3)-(P4).

\begin{prop}[Preservation of the particles' sign in cases (P3)-(P4)]
Assume (V), (I) and (P) are satisfied and, moreover, assume (V$^*$) is satisfied in case (P4). Then, as long as the solution to \eqref{3} and \eqref{4} exists, we have that
\begin{equation}
\label{8}
x_{i}(t)\leq 0 \ \ \text{for} \ i\in \{0,\dots,k_{n}-1\} \quad \text{and} \quad x_{i}(t)\geq 0 \ \ \text{for} \ i\in \{k_{n}+1,\dots,n\}.
\end{equation}
For $i=k_{n}$, we have that
\begin{itemize}
\item[(a)] If $\overline{x}_{k_n}<0$, then $x_{k_n}(t)\leq 0$;
\item[(b)] If $\overline{x}_{k_n}=0$, then $x_{k_n}(t)\equiv 0$.
\end{itemize}
\end{prop}
\begin{proof}
Let us first show \eqref{8} for a fixed $i\in \{0,\dots,k_{n}-1\}$ (for the remaining indices $i\in \{k_{n}+1,\dots,n\}$ we can follow a symmetric reasoning) and let us suppose that the solution to \eqref{3} and \eqref{4} exists on the time interval $[0,\tau)$ for a certain $\tau>0$ (possibly infinite): we have to prove that $x_{i}(t)\leq 0$ for all $t\in [0,\tau)$. Suppose by contradiction the existence of a time $0<t^*<\tau$ such that $x_{i}(t^*)>0$. Since $x_{i}(0)=\overline{x}_{i}<0$ and $x_{i}(t)$ is a continuous function, then the pre-image of $\{0\}$ via $x_i(\cdot)$ is a compact set in $(0,t^*)$. Therefore, there exist two positive times (possibly coinciding) $0<t_{1}\leq t_{2}<t^*$ such that
\begin{equation}
\label{520}
x_{i}(t_{1})=x_{i}(t_{2})=0, \quad x_{i}(t)<0 \ \text{for all} \ 0\leq t<t_{1} \quad \text{and} \ \quad x_{i}(t)>0 \ \text{for all} \ t_{2}< t\leq t^*.
\end{equation}
Let us now split the cases (P3) and (P4). Concerning the case (P3), from the first inequality in \eqref{520} we get that $\dot{x}_{i}(t)=v(R_{i-1}(t))\phi(x_{i}(t))\leq 0$ for all $0\leq t< t_{1}$, which implies
\begin{equation*}
x_{i}(t_{1})=\overline{x}_{i}+\int_{0}^{t_{1}}\dot{x}_{i}(t)dt<0
\end{equation*}
and this contradicts the fact that $x_{i}(t_{1})=0$. Turning to case (P4), using the second inequality in \eqref{520}, it follows that $\dot{x}_{i}(t)=v(R_{i}(t))\phi(x_{i}(t))\leq 0$ for all $t_{2}< t\leq t^{*}$, hence we get
\begin{equation*}
x_{i}(t^*)=x_{i}(t_{2})+\int_{t_{2}}^{t^*}\dot{x}_{i}(t)dt\leq 0,
\end{equation*}
which contradicts the fact that $x_{i}(t^*)>0$.

Now it remains to prove (b), indeed in case $\overline{x}_{k_n}<0$ we can repeat the same argument exposed above to show the validity of \eqref{8}. Let us now consider the cases (P3) and (P4) separately. Starting from case (P3), we first notice that $|v(R_{k_n-1}(t))\phi(x_{k_n}(t))|\leq v_{\max}|\phi(x_{k_n}(t))|$ for all $t\geq 0$. As a consequence, by comparison we have that the unique solutions to the two Cauchy problems
\begin{equation}
\label{523}
\begin{cases}
\dot{y}_{0}(t)=v_{\max}\phi(y_{0}(t)),
\\
\dot{y}_{i}(t)=v(R_{i-1}(t))\phi(y_{i}(t)), \ i\in\{1,\dots, k_n-1\}
\\
\dot{y}_{k_n}(t)=-v_{\max}|\phi(y_{k_n}(t))|,
\\
\dot{y}_{i}(t)=v(R_{i}(t))\phi(y_{i}(t)), \ i\in\{k_n+1,\dots, n-1\}
\\
\dot{y}_{n}(t)=v_{\max}\phi(y_{n}(t)),
\\
y_{i}(0)=\overline{x}_{i}, \ i\in\{0,\dots, n\}
\end{cases}
\!\!\!\!\! \text{and} \
\begin{cases}
\dot{z}_{0}(t)=v_{\max}\phi(z_{0}(t)),
\\
\dot{z}_{i}(t)=v(R_{i-1}(t))\phi(z_{i}(t)), \ i\in\{1,\dots, k_n-1\}
\\
\dot{z}_{k_n}(t)=v_{\max}|\phi(z_{k_n}(t))|,
\\
\dot{z}_{i}(t)=v(R_{i}(t))\phi(z_{i}(t)), \ i\in\{k_n+1,\dots, n-1\}
\\
\dot{z}_{n}(t)=v_{\max}\phi(z_{n}(t)),
\\
z_{i}(0)=\overline{x}_{i}, \ i\in\{0,\dots, n\}
\end{cases}
\end{equation}
are such that
\begin{equation*}
y_{i}(t)\leq x_{i}(t)\leq z_{i}(t) \quad \text{for all} \ t\in [0,\tau^*), \ i\in\{0,\dots,n\}, \quad \text{with} \ \tau^*:=\sup\{t\leq \tau: \text{the solutions to}\ \eqref{523}\ \text{exist}\}.
\end{equation*}
To see this, we observe that neither in \eqref{3} nor in \eqref{523} does the $(k_n)-$th particle affect the evolution of the other particles (which remains unchanged since no particle sees the $(k_n)-$th particle), hence it holds that $y_{i}(t)=z_{i}(t)=x_{i}(t)$ for all $t\in [0,\tau^{*})$ and $i\neq k_{n}$. Therefore, the evolution of the $(k_n)-$th particle is completely decoupled in both systems \eqref{523} and the choice of the velocities $\dot{y}_{k_n}$ and $\dot{z}_{k_n}$ easily implies the assertion. Now, since $y_{k_n}(0)=z_{k_n}(0)=0$ and since $\dot{y}_{k_n}(t)=-v_{\max}|\phi(y_{k_n}(t))|$ and $\dot{z}_{k_n}(t)=v_{\max}|\phi(z_{k_n}(t))|$ have the stationary solution $y_{k_n}(t)\equiv z_{k_n}(t)\equiv 0$, then it follows that $x_{k_n}(t)\equiv 0$ for all $t\in [0,\tau^*)$. Moreover, the only possible reason for which $\tau^*<\tau$ is that either $y_{k_n}$ hits $y_{k_n-1}$ or that $z_{k_n}$ hits $z_{k_n+1}$ on some $t<\tau^*$, which cannot happen due to $y_{k_n}(t)\equiv z_{k_n}(t)\equiv 0$ on $t\in [0,\tau^*)$. Hence, $\tau^*\geq \tau$ and (b) is proven.
\\
Concerning case (P4), we first show that $x_{k_n}(t)\leq 0$ for all $t\in [0,\tau)$ arguing again by contradiction and in particular supposing the existence of a time $0<t^*<\tau$ such that $x_{k_n}(t^*)>0$. Then there exists $0\leq \tilde{t}< t^*$ such that $x_{k_{n}}(\tilde{t})=0$ and $x_{k_n}(t)>0$ for all $\tilde{t}<t\leq t^*$. This implies that $\dot{x}_{k_{n}}(t)\leq 0$ for all $\tilde{t}<t\leq t^*$ and hence that
\begin{equation*}
x_{k_n}(t^*)=x_{k_n}(\tilde{t})+\int_{\tilde{t}}^{t^*}\dot{x}_{k_n}(t)dt\leq 0
\end{equation*}
which contradicts the fact that $x_{k_n}(t^*)>0$. Reasoning in a symmetric way, that is supposing by contradiction the existence of a time $0<t^*<\tau$ such that $x_{k_n}(t)<0$, we can show that $x_{k_n}(t)\geq 0$ for all $t\in [0,\tau)$, hence we finally get $x_{k_n}(t)\equiv 0$ for all $t\in [0,\tau)$.
\end{proof}

The next proposition ensures that particles never collide in all the four cases: this gives the global existence of the solution for \eqref{1}-\eqref{4}.

\begin{prop}[Discrete maximum principle]
\label{prop1}
Assume that (V), (I) and (P) are satisfied.
\begin{itemize}
\item If (P1) (respectively (P2), (P3)) holds then, as long as it exists, the solution to \eqref{1} (respectively to \eqref{2}, \eqref{3}) satisfies 
    \begin{equation}
\label{53}
x_{i+1}^{n}(t)-x_{i}^{n}(t)\geq \frac{\ell_{n}}{\overline{R}}e^{-L't} \quad \text{for} \ i\in \{0,\dots,n-1\}.
\end{equation}
\item If (V$^*$) and (P4) hold then, as long as it exists, the solution to \eqref{4} satisfies
\begin{equation}
\label{15}
x_{i+1}^{n}(t)-x_{i}^{n}(t)\geq \frac{\ell_{n}}{R_{\max}} \quad \text{for} \  i\in \{0,\dots,n-1\}.
\end{equation}
\end{itemize}
\end{prop}
\begin{proof}
We first observe that the statement is true for $t=0$, indeed by \eqref{5} and (I) it follows that
\begin{equation*}
\overline{x}_{i+1}-\overline{x}_{i}\geq \frac{\ell_{n}}{\overline{R}}\geq \frac{\ell_{n}}{R_{\max}} \quad \text{for all}\ i\in \{0,\dots,n-1\}.
\end{equation*}
We now consider the four cases separately, each one with a separate recursive argument.

Let us suppose (P1) holds and let us take as basis of our (backward) recursive argument the index $i=n-1$. Since the first order Taylor's expansion of $\phi$ at $x_{n-1}(t)$ is given by
\begin{equation*}
\phi(x_{n}(t))=\phi(x_{n-1}(t))+\phi'(\tilde{x}(t))(x_{n}(t)-x_{n-1}(t))
\end{equation*}
for some $\tilde{x}(t)\in \bigl(x_{n-1}(t),x_{n}(t)\bigr)$, then it follows that
\begin{equation*}
\begin{split}
\frac{d}{dt}\bigl[x_{n}(t)-x_{n-1}(t)\bigr]=&v_{\max}\phi(x_{n}(t))-v(R_{n-1}(t))\phi(x_{n-1}(t))
\\
=&\bigl(v_{\max}-v(R_{n-1}(t))\bigr)\phi(x_{n-1}(t))+v_{\max}\phi'(\tilde{x}(t))\bigl(x_{n}(t)-x_{n-1}(t)\bigr)
\\
\geq& v_{\max}\phi'(\tilde{x}(t))\bigl(x_{n}(t)-x_{n-1}(t)\bigr).
\end{split}
\end{equation*}
Therefore, applying Gronwall lemma, we get that
\begin{equation*}
x_{n}(t)-x_{n-1}(t)\geq \bigl(\overline{x}_{n}-\overline{x}_{n-1}\bigr)e^{\int_{0}^{t}v_{\max}\phi'(\tilde{x}(s))ds}\geq \frac{\ell_{n}}{\overline{R}} e^{\int_{0}^{t}v_{\max}\phi'(\tilde{x}(s))ds}
\end{equation*}
and finally, since the assumption (P) implies that $\phi'$ is bounded, from the previous inequality we easily get \eqref{53} for $i=n-1$.
\\
Concerning the remaining indices, we argue by contradiction and assume (without restriction) the existence of an index $j\in \{0,\dots, n-2\}$ and of a time $t^*>0$ such that
\begin{equation*}
x_{i+1}(t)-x_{i}(t)\geq \frac{\ell_{n}}{\overline{R}}e^{-L't} \quad \text{for} \ t\geq 0,\  i\in \{j+1,\dots,n-1\}
\end{equation*}
and
\begin{equation*}
x_{j+1}(t^*)-x_{j}(t^*)< \frac{\ell_{n}}{\overline{R}}e^{-L't^*}.
\end{equation*}
Since $f(t):=x_{j+1}(t)-x_{j}(t)$ and $\displaystyle g(t):=\frac{\ell_{n}}{\overline{R}}e^{-L't}$ are continuous functions satisfying $f(0)\geq g(0)$ and $f(t^*)<g(t^*)$, then by continuity there exists (at least) a time at which $f$ and $g$ coincide. As a consequence, we can define
\begin{equation*}
\tilde{t}:=\sup\{0\leq\tau < t^*: f(\tau)=g(\tau)\}
\end{equation*}
and we can claim that $x_{j+1}(t)-x_{j}(t)<\displaystyle\frac{\ell_{n}}{\overline{R}}e^{-L't}$ for $\tilde{t}<t\leq t^*$. Suppose indeed by contradiction the existence of $t_{1}\in (\tilde{t},t^*)$ such that $f(t_{1})\geq g(t_{1})$, then we can use again the previous continuity argument to get the existence of $t_{2}\in [t_{1},t^*)$ such that $f(t_{2})=g(t_{2})$ and this contradicts the fact that $\tilde{t}$ is the sup of the times with this property. As a consequence, for $t\in[\tilde{t},t^{*}]$ we have the following situation
\begin{equation}
\label{7}
x_{j+1}(t)-x_{j}(t)\begin{cases}
 = \frac{\ell_{n}}{\overline{R}}e^{-L'\tilde{t}} & \quad \text{for} \ t=\tilde{t},
\\
< \frac{\ell_{n}}{\overline{R}}e^{-L't} & \quad \text{for} \ \tilde{t}<t\leq t^*.
\end{cases}
\end{equation}
Using as before the first order Taylor's expansion of $\phi$ at $x_{j}(t)$, it holds
\begin{equation}
\label{100}
\begin{split}
\frac{d}{dt}\bigl[x_{j+1}(t)-x_{j}(t)\bigr]=&v(R_{j+1}(t))\phi(x_{j+1}(t))-v(R_{j}(t))\phi(x_{j}(t))
\\
=&\bigl(v(R_{j+1}(t))-v(R_{j}(t))\bigr)\phi(x_{j}(t))+v(R_{j+1}(t))\phi'(\tilde{x}(t))\bigl(x_{j+1}(t)-x_{j}(t)\bigr)
\end{split}
\end{equation}
for some $\tilde{x}(t)\in \bigl(x_{j}(t),x_{j+1}(t)\bigr)$. On the other hand, from the contradictory assumption it follows that
\begin{equation*}
R_{j+1}(t)\leq \overline{R}e^{L't}<R_{j}(t)
\end{equation*}
and hence, since $v$ is non-increasing, that
\begin{equation*}
v(R_{j}(t))\leq v\bigl(\overline{R}e^{L't}\bigr)\leq v(R_{j+1}(t))
\end{equation*}
for all $\tilde{t}<t\leq t^*$. Using these estimates in \eqref{100}, for all $\tilde{t}<t\leq t^*$ we get that
\begin{equation*}
\frac{d}{dt}\bigl[x_{j+1}(t)-x_{j}(t)\bigr]\geq v(R_{j+1}(t))\phi'(\tilde{x}(t))\bigl(x_{j+1}(t)-x_{j}(t)\bigr)
\end{equation*}
and finally, using again Gronwall lemma on the time interval $(\tilde{t},t)$ with $t\leq t^*$, it follows that
\begin{equation*}
x_{j+1}(t)-x_{j}(t)\geq \bigl(x_{j+1}(\tilde{t})-x_{j}(\tilde{t})\bigr)e^{\int_{\tilde{t}}^{t}v(R_{j+1}(s))\phi'(\tilde{x}(s))ds}\geq \frac{\ell_{n}}{\overline{R}}e^{-L't},
\end{equation*}
which contradicts \eqref{7}, since $\tilde{t}<t\leq t^*$.

Concerning case (P2), we can reason in a symmetric way, taking the index $i=0$ as basis of the recursive argument and assuming the existence of a 'first index $j$' at which the statement fails on a certain time interval (we omit the details).

In case (P3) holds, the proof of \eqref{53} for $i\in \{0,\dots,k_{n}-1\}$ and $i\in \{k_{n}+1,\dots,n\}$ is straightforward, indeed we can simply apply (rearranging the indices properly) the same arguments used in cases (P2) and (P1) respectively. Turning to the remaining case $i=k_{n}$, we first remark that
\begin{equation*}
v(R_{k_{n}-1}(t))\geq v(\overline{R}e^{L't}) \quad \text{and} \quad v(R_{k_{n}+1}(t))\geq v(\overline{R}e^{L't}) \ \text{for all}\ t\geq0,
\end{equation*}
since \eqref{53} is valid for $i=k_{n}-1$ and $i=k_{n}+1$.
As a consequence we have
\begin{equation*}
\begin{split}
\frac{d}{dt}\bigl[x_{k_{n}+1}(t)\!-\!x_{k_{n}}(t)\bigr] &\! =\! v(R_{k_{n}+1}(t))\phi(x_{k_{n}+1}(t))-v(R_{k_{n}-1}(t))\phi(x_{k_{n}}(t))\!\geq\! v(\overline{R}e^{L't})\bigl(\phi(x_{k_{n}+1}(t))-\phi(x_{k_{n}}(t))\bigr)
\end{split}
\end{equation*}
and, using the first order Taylor's expansion of $\phi$, it follows that
\begin{equation*}
\frac{d}{dt}\bigl[x_{k_{n}+1}(t)-x_{k_{n}}(t)\bigr] \geq v(\overline{R}e^{L't})\phi'(\tilde{x}(t))\bigl(x_{k_{n}+1}(t)-x_{k_{n}}(t)\bigr)
\end{equation*}
for some $\tilde{x}(t)\in \bigl(x_{k_{n}}(t),x_{k_{n}+1}(t)\bigr)$.
Finally, applying again Gronwall lemma we get
\begin{equation*}
x_{k_{n}+1}(t)-x_{k_{n}}(t)\geq \bigl(\overline{x}_{k_{n}+1}-\overline{x}_{k_{n}}\bigr)e^{\int_{0}^{t}v(\overline{R}e^{Ls})\phi'(\tilde{x}(s))ds}\geq \frac{\ell_{n}}{\overline{R}}e^{-L't}
\end{equation*}
and this concludes the proof for case (P3).

In case (P4), we prove \eqref{15} using a different recursive argument with respect to the previous cases. Let us first consider the indices $i\in \{0,\dots,k_{n}-1\}$, let us take as base case the index $i=k_{n}-1$ and suppose by contradiction the existence of a time $t^*>0$ such that $x_{k_{n}}(t^*)-x_{k_{n}-1}(t^*)<\frac{\ell_{n}}{R_{\max}}$. Using a similar continuity argument with respect to the proof of case (P1), we can prove the existence of $0\leq \tilde{t}<t^*$ such that
\begin{equation*}
x_{k_{n}}(t)-x_{k_{n}-1}(t)\begin{cases}& = \frac{\ell_{n}}{R_{\max}} \quad \text{for} \ t=\tilde{t},
\\
& < \frac{\ell_{n}}{R_{\max}} \quad \text{for} \ \tilde{t}<t\leq t^*.
\end{cases}
\end{equation*}
Integrating the ODE in \eqref{4}, for all $\tilde{t}<t\leq t^*$ it follows that
\begin{equation*}
\begin{split}
x_{k_{n}}(t)-x_{k_{n}-1}(t)=&x_{k_{n}}(\tilde{t})-x_{k_{n}-1}(\tilde{t})+\int_{\tilde{t}}^{t}\bigl[v(R_{k_{n}}(s))\phi({x}_{k_{n}}(s))-v(R_{k_{n}-1}(s))\phi(x_{k_{n}-1}(s))\bigr]ds,
\end{split}
\end{equation*}
where, due to the contradictory assumption, we have that $R_{k_{n}-1}(s)> R_{\max}$ and this implies, together with (V$^*$), that $v(R_{k_{n}-1}(s))=0$. Hence we get, for $\tilde{t}<t\leq t^*$, that
\begin{equation*}
\int_{\tilde{t}}^{t}\bigl[v(R_{k_{n}}(s))\phi({x}_{k_{n}}(s))-v(R_{k_{n}-1}(s))\phi(x_{k_{n}-1}(s))\bigr]ds=\int_{\tilde{t}}^{t}v(R_{k_{n}}(s))\phi({x}_{k_{n}}(s))ds\geq 0
\end{equation*}
and so it follows that
\begin{equation*}
x_{k_{n}}(t)-x_{k_{n}-1}(t)\geq x_{k_{n}}(\tilde{t})-x_{k_{n}-1}(\tilde{t}) = \frac{\ell_{n}}{R_{\max}},
\end{equation*}
which is a contradiction. For the remaining indices $i\in \{0,\dots,k_{n}-2\}$, we can repeat a recursive argument similar to case (P1), while
the validity of \eqref{15} for $i\in \{k_{n}+1,\dots,n-1\}$ can be proved in a symmetric way with respect to the previous indices: it is sufficient to take as base case the index $i=k_{n}+1$ and then proceeding by contradiction (the details are left to the reader). As a consequence, it remains to show \eqref{15} for $i=k_{n}$ and for this purpose we argue again by contradiction, supposing the existence of a time $t^*>0$ such that $x_{k_{n}}(t^*)-x_{k_{n}-1}(t^*)<\frac{\ell_{n}}{R_{\max}}$, which implies as before the existence of $0\leq \tilde{t}<t^*$ such that
\begin{equation*}
x_{k_{n}+1}(t)-x_{k_{n}}(t)\begin{cases}
& = \frac{\ell_{n}}{R_{\max}} \quad \text{for} \ t=\tilde{t},
\\
& < \frac{\ell_{n}}{R_{\max}} \quad \text{for} \ \tilde{t}<t\leq t^*.
\end{cases}
\end{equation*}
On the other hand, for all $\tilde{t}<t\leq t^*$ it holds that
\begin{equation*}
x_{k_{n}+1}(t)-x_{k_{n}}(t)=x_{k_{n}+1}(\tilde{t})-x_{k_{n}}(\tilde{t})+\int_{\tilde{t}}^{t}v(R_{k_{n}}(s))\bigl[\phi({x}_{k_{n}+1}(s))-\phi(x_{k_{n}}(s))\bigr]ds,
\end{equation*}
where the contradictory assumption implies that $R_{k_{n}}(s)> R_{\max}$ and hence, due to (V$^*$), that $v(R_{k_{n}}(s))=0$ for all $\tilde{t}<t\leq t^*$. From this it follows that
\begin{equation*}
x_{k_{n}+1}(t)-x_{k_{n}}(t)=x_{k_{n}+1}(\tilde{t})-x_{k_{n}}(\tilde{t})=\frac{\ell_{n}}{R_{\max}}\quad \text{for}\ \tilde{t}<t\leq t^*,
\end{equation*}
which is a contradiction and this concludes the proof for case (P4).
\end{proof}

\begin{rem}
The exponential rate in \eqref{53} is not optimal: using the same strategy, we can indeed prove that
\begin{equation*}
x_{i+1}(t)-x_{i}(t)\geq  \frac{\ell_{n}}{\overline{R}}e^{v_{\max}\phi'_{\inf}t}\quad \text{if}\ \phi'_{\inf}:=\inf_{x\in\mathbb{R}}\phi'(x)\leq 0
\end{equation*}
and
\begin{equation*}
x_{i+1}(t)-x_{i}(t)\geq  \frac{\ell_{n}}{\overline{R}}e^{v_{\min}\phi'_{\inf}t}\quad \text{if}\ \phi'_{\inf}\geq 0, \ \text{with} \ v_{\min}:=\min_{\rho\in \mathbb{R}^{+}} v(\rho).
\end{equation*}
\end{rem}

\begin{rem}
\label{rem1}
The estimates proven in Lemma \ref{lem7} are global-in-time in view of the discrete maximum principle above.
\end{rem}

\begin{rem}
We can slightly weaken our assumption (P) in three of the cases examined, indeed it is sufficient to take $\phi\in W^{2,\infty}([\overline{x}_{\min},+\infty))$ in case of forward movement, $\phi\in W^{2,\infty}((-\infty,\overline{x}_{\max}])$ in case of backward movement and $\phi\in W^{2,\infty}([\overline{x}_{\min},\overline{x}_{\max}])$ in case of attractive movement.
\end{rem}

According to our construction and the previous propositions and remarks, the evolution of the $n+1$ particles $x_{i}(t)$ is well defined for all $t\geq 0$, hence we can introduce a time-depending piecewise constant density having support in $[x_{0}(t),x_{n}(t)]$. Therefore, we set
\begin{equation}
\label{10}
\rho^{n}(x,t):=\sum_{i=0}^{n-1}R_{i}(t)\mathbbm{1}_{[x_{i}(t),x_{i+1}(t))}(x)=\sum_{i=0}^{n-1}\frac{\ell_{n}}{x_{i+1}(t)-x_{i}(t)}\mathbbm{1}_{[x_{i}(t),x_{i+1}(t))}(x).
\end{equation}

\section{BV estimate, time continuity and compactness}\label{sec:compactness}
We first show a uniform control of the total variation of $\rho^{n}$ which plays a key role in the proof of the convergence of our particle scheme. In the sequel, since we are interested on large values of $n\in \mathbb{N}$, without loss of generality we suppose $n$ sufficiently large such that $2\leq k_{n}\leq n-3$.

\begin{prop}
\label{prop2}
Assume (V), (I) and (P) are satisfied and, moreover, assume (V$^*$) is satisfied in case (P4). If one of (P1), (P2), (P3) or (P4) holds, then there exist four positive constants $\alpha$, $\beta$, $\gamma$ and $\zeta$, independent on $n$, such that
\begin{equation}
\label{11}
\TV[\rho^{n}(\cdot,t)]\leq \bigl(\TV[\overline{\rho}]+\alpha t+\beta e^{L't}\bigr) e^{[\gamma t(1+t)+\zeta e^{L't}]} \quad \text{for all}\ \ t\geq 0.
\end{equation}
\end{prop}
\begin{proof}
We first notice that
\begin{equation*}
\TV[\rho^{n}(\cdot,t)]=R_{0}(t)+R_{n-1}(t)+\sum_{i=0}^{n-2}\bigl|R_{i}(t)-R_{i+1}(t)\bigr|.
\end{equation*}
Now we would like to derive the previous identity: since the absolute value function is not derivable in $z=0$, we need to introduce one of its $C^1$ approximation, for instance we consider
\begin{equation*}
\eta_{\sigma}(z):=\begin{cases} |z| \quad &\text{if} \ |z|\geq \sigma,
\\
\frac{z^2}{2\sigma}+\frac{\sigma}{2} \quad &\text{if} \ |z|< \sigma,
\end{cases}
\quad \text{for some} \ \sigma>0.
\end{equation*}
We first notice that $\eta_{\sigma}(z)\geq 0$ for all $z\in \mathbb{R}$. Moreover, since
\begin{equation*}
\eta_{\sigma}'(z)=\begin{cases} \sign(z) \quad &\text{if} \ |z|\geq\sigma,
\\
\frac{z}{\sigma} \quad &\text{if} \ |z|<\sigma,
\end{cases}
\end{equation*}
then it follows that $|\eta_{\sigma}'(z)|\leq 1$ and
\begin{equation}
\label{987}
0\leq \eta'_{\sigma}(z)z\leq \eta_{\sigma}(z) \quad \text{for all} \ z\in \mathbb{R},
\end{equation}
in particular it holds that $|\eta_{\sigma}'(z)|= 1$ and $\eta'_{\sigma}(z)z=\eta_{\sigma}(z)$ for $|z|\geq \sigma$, while for $|z|<\sigma$ we have that $\displaystyle|\eta_{\sigma}'(z)|=\frac{|z|}{\sigma}<1$ and
\begin{equation*}
\eta'_{\sigma}(z)z=\frac{z^2}{\sigma}< \frac{z^2}{2\sigma}+\frac{\sigma}{2}=\eta_{\sigma}(z).
\end{equation*}
Then, defining
\begin{equation}
\label{122}
\TV_{\sigma}[\rho^{n}(\cdot,t)]:=R_{0}(t)+R_{n-1}(t)+\sum_{i=0}^{n-2}\eta_{\sigma}\bigl(R_{i}(t)-R_{i+1}(t)\bigr),
\end{equation}
it follows that
\begin{equation*}
\frac{d}{dt}\TV_{\sigma}[\rho^{n}(\cdot,t)]=\dot{R}_{0}(t)+\dot{R}_{n-1}(t)+\sum_{i=0}^{n-2}\eta_{\sigma}'\bigl(R_{i}(t)-R_{i+1}(t)\bigr)\bigl[\dot{R}_{i}(t)-\dot{R}_{i+1}(t)\bigr],
\end{equation*}
where we can rewrite the sum as
\begin{equation*}
\begin{split}
\sum_{i=0}^{n-2}\eta_{\sigma}'\bigl(R_{i}(t)-R_{i+1}(t)\bigr)\bigl[\dot{R}_{i}(t)-\dot{R}_{i+1}(t)\bigr]=&\eta_{\sigma}'\bigl(R_{0}(t)-R_{1}(t)\bigr)\dot{R}_{0}(t)-\eta_{\sigma}'\bigl(R_{n-2}(t)-R_{n-1}(t)\bigr)\dot{R}_{n-1}(t)
\\
&+\sum_{i=1}^{n-2}\bigl[\eta_{\sigma}'\bigl(R_{i}(t)-R_{i+1}(t)\bigr)-\eta_{\sigma}'\bigl(R_{i-1}(t)-R_{i}(t)\bigr)\bigr]\dot{R}_{i}(t).
\end{split}
\end{equation*}
Moreover, defining for brevity
\begin{equation*}
\begin{split}
&\mu_{0}(t):=1+\eta_{\sigma}'\bigl(R_{0}(t)-R_{1}(t)\bigr),
\\
&\mu_{n-1}(t):=1-\eta_{\sigma}'\bigl(R_{n-2}(t)-R_{n-1}(t)\bigr),
\\
&\mu_{i}(t):=\eta_{\sigma}'\bigl(R_{i}(t)-R_{i+1}(t)\bigr)-\eta_{\sigma}'\bigl(R_{i-1}(t)-R_{i}(t)\bigr) \quad \text{for}\ i=1,\dots,n-2,
\\
&\omega_{i}(t):=\eta_{\sigma}'\bigl(R_{i}(t)-R_{i+1}(t)\bigr)\quad \hspace{3.5cm}\text{for}\ i=0,\dots,n-1,
\end{split}
\end{equation*}
we hence get that
\begin{equation}
\label{126}
\frac{d}{dt}\TV_{\sigma}[\rho^{n}(\cdot,t)]= A(t)+B(t)+C(t),
\end{equation}
with
\begin{equation*}
A(t):= \mu_{0}(t)\dot{R}_{0}(t), \quad B(t):= \mu_{n-1}(t)\dot{R}_{n-1}(t) \quad \text{and} \quad C(t):= \sum_{i=1}^{n-2}\mu_{i}(t)\dot{R}_{i}(t).
\end{equation*}
Now we need to determine an upper bound for the functions $A(t)$, $B(t)$, $C(t)$ and we should treat as before the four cases separately.

Let us consider the first case (P1). For $A(t)$ and $B(t)$ we have, due to (P) and \eqref{53}, the following sub-cases:
\begin{equation}
A(t)\begin{cases} =0 \! &\text{if}\
R_{0}-R_{1}\leq -\sigma,
\\
\leq 2L' \overline{R} e^{L't}+ 2 \frac{\overline{R}^2\Lip_{v}||\phi||_{L^{\infty}(\mathbb{R})}}{\ell_{n}}\sigma\! &\text{if}\ |R_{0}-R_{1}|< \sigma,
\\
\leq 2L' \overline{R} e^{L't} \! &\text{if}\ R_{0}-R_{1}\geq \sigma
\end{cases}
\end{equation}
and
\begin{equation}
B(t)\begin{cases} =0 \ &\text{if}\ R_{n-2}-R_{n-1}\geq \sigma,
\\
\leq 2L' \overline{R} e^{L't} \ &\text{if}\ R_{n-2}-R_{n-1} <\sigma.
\end{cases}
\end{equation}
Let us first prove the estimates for $A(t)$, where we recall that
\begin{equation*}
A(t)=\bigl[1+\eta'_{\sigma}(R_{0}-R_{1})\bigr]\frac{R_{0}^2}{\ell_{n}}\bigl[v(R_{0})\phi(x_{0})-v(R_{1})\phi(x_{1})\bigr].
\end{equation*}
We have that:
\begin{itemize}
\item[(a)] If $R_{0}-R_{1}\leq -\sigma$, then $\eta'_{\sigma}(R_{0}-R_{1})=-1$ and so $A(t)=0$.
\item[(b)] If $R_{0}-R_{1}\geq \sigma$, then $\eta'_{\sigma}(R_{0}-R_{1})=1$, $v(R_{0})\leq v(R_{1})$ and so it holds that
\begin{equation*}
A(t)\leq 2\frac{R_{0}^2}{\ell_{n}}v(R_{1})\bigl[\phi(x_{0})-\phi(x_{1})\bigr]\leq 2L' \overline{R} e^{L't}.
\end{equation*}
\item[(c)] If $|R_{0}-R_{1}|<\sigma$, then $-1<\eta'_{\sigma}(R_{0}-R_{1})< 1$ and hence
\begin{equation*}
\begin{split}
A(t)=& \bigl[1+\eta'_{\sigma}(R_{0}-R_{1})\bigr]\frac{R_{0}^2}{\ell_{n}}v(R_{0})\bigl[\phi(x_{0})-\phi(x_{1})\bigr]+ \bigl[1+\eta'_{\sigma}(R_{0}-R_{1})\bigr]\frac{R_{0}^2}{\ell_{n}}\phi(x_{1})\bigl[v(R_{0})-v(R_{1})\bigr]
\\
\leq & 2L' \overline{R} e^{L't}+ 2 \frac{\overline{R}^2||\phi||_{L^{\infty}(\mathbb{R})}}{\ell_{n}}\bigl|v(R_{0})-v(R_{1})\bigr|\leq 2L' \overline{R} e^{L't}+ 2 \frac{\overline{R}^2\Lip_{v}||\phi||_{L^{\infty}(\mathbb{R})}}{\ell_{n}}\sigma.
\end{split}
\end{equation*}
\end{itemize}
Turning to $B(t)$, where
\begin{equation*}
B(t)= \bigl[1-\eta_{\sigma}'\bigl(R_{n-2}-R_{n-1}\bigr)\bigr]\frac{R_{n-1}^2}{\ell_{n}}\bigl[v(R_{n-1})\phi(x_{n-1})-v_{\max}\phi(x_{n})\bigr],
\end{equation*}
it holds that:
\begin{itemize}
\item[(a)] If $R_{n-2}-R_{n-1}\geq \sigma$, then $\eta'_{\sigma}(R_{n-2}-R_{n-1})=1$ and so $B(t)=0$.
\item[(b)] If $R_{n-2}-R_{n-1}< \sigma$, then $-1\leq \eta'_{\sigma}(R_{n-2}-R_{n-1})<1$ and so we get that
\begin{equation*}
B(t)\leq 2\frac{R_{n-1}^2}{\ell_{n}}v_{\max}\bigl[\phi(x_{n-1})-\phi(x_{n})\bigr]\leq 2L' \overline{R} e^{L't}.
\end{equation*}
\end{itemize}
Regarding the sum $C(t)$, we can split it as
\begin{equation*}
\begin{split}
C(t) =\sum_{i=1}^{n-2}I_{i}(t)+\sum_{i=1}^{n-2}\mu_{i}(t)R_{i}(t)I\! I_{i}(t),
\end{split}
\end{equation*}
where
\begin{equation*}
I_{i}(t):=\mu_{i}(t)\frac{R_{i}(t)^{2}}{\ell_{n}}\phi(x_{i+1}(t))\bigl[v(R_{i}(t))-v(R_{i+1}(t))\bigr] \quad \text{and}\quad I\! I_{i}(t):= \frac{R_{i}(t)}{\ell_{n}}v(R_{i}(t))\bigl[\phi(x_{i}(t))-\phi(x_{i+1}(t))\bigr].
\end{equation*}
Now we show that
\begin{equation*}
I_{i}(t)\leq \frac{c}{\ell_{n}}\sigma \quad \text{for all} \ i\in \{1,\dots,n-2\},
\end{equation*}
where $c$ is a non-negative constant depending only on $\overline{\rho}$, $\phi$ and $v$. In particular we have that:
\begin{itemize}
\item[(a)] If $R_{i}-R_{i+1}\geq \sigma$ and $R_{i-1}-R_{i}\geq \sigma$ or if $R_{i}- R_{i+1}\leq -\sigma$ and $R_{i-1}- R_{i}\leq -\sigma$, then $\mu_{i}(t)=0$ and so $I_{i}(t)=0$.
\item[(b)] If $R_{i}-R_{i+1}\geq \sigma$ and $R_{i-1}-R_{i}\leq -\sigma$, then $\mu_{i}(t)=2$, $v(R_{i})\leq v(R_{i+1})$ and hence
\begin{equation*}
I_{i}(t)=2\frac{R_{i}^{2}}{\ell_{n}}\phi(x_{i+1})\bigl[v(R_{i})-v(R_{i+1})\bigr]\leq 0.
\end{equation*}
\item[(c)] If $R_{i}-R_{i+1}\leq -\sigma$ and $R_{i-1}-R_{i}\geq \sigma$, then $\mu_{i}(t)=-2$, $v(R_{i+1})\leq v(R_{i})$ and so as before $I_{i}(t)\leq 0$.
\item[(d)] If $R_{i}-R_{i+1}\geq \sigma$ and $|R_{i-1}-R_{i}|< \sigma$, then $0<\mu_{i}(t)<2$,  $v(R_{i})\leq v(R_{i+1})$ and so again $I_{i}(t)\leq 0$.
\item[(e)] If $R_{i}-R_{i+1}\leq -\sigma$ and $|R_{i-1}-R_{i}|< \sigma$, then $-2<\mu_{i}(t)<0$,  $v(R_{i+1})\leq v(R_{i})$ and hence $I_{i}(t)\leq 0$.
\item[(f)] If $|R_{i}-R_{i+1}|< \sigma$ and $R_{i-1}-R_{i}\leq -\sigma$, then $0<\mu_{i}(t)<2$ and therefore
\begin{equation*}
I_{i}(t)\leq 2\frac{\overline{R}^{2}||\phi||_{L^{\infty}(\mathbb{R})}}{\ell_{n}}\bigl|v(R_{i})-v(R_{i+1})\bigr| \leq 2\frac{\overline{R}^{2}\Lip_{v}||\phi||_{L^{\infty}(\mathbb{R})}}{\ell_{n}}\sigma.
\end{equation*}
\item[(g)] If $|R_{i}-R_{i+1}|< \sigma$ and $R_{i-1}-R_{i}\geq \sigma$, then $-2<\mu_{i}(t)<0$ and hence it holds the same estimate of the previous case.
\item[(h)] If $|R_{i}-R_{i+1}|< \sigma$ and $|R_{i-1}-R_{i}|< \sigma$, then $-2<\mu_{i}(t)<2$ and again we have that $I_{i}(t)$ satisfies the same estimate as in case (f).
\end{itemize}
As a consequence, we have that
\begin{equation}
C(t) \leq \frac{c}{\ell_{n}^2}\sigma+\sum_{i=1}^{n-2}\mu_{i}(t)R_{i}(t)I\! I_{i}(t)
\end{equation}
and we can rewrite the sum as
\begin{equation}
\label{eq24}
\begin{split}
\sum_{i=1}^{n-2}\mu_{i}(t)R_{i}(t)I\! I_{i}(t)=&\sum_{i=1}^{n-2}\omega_{i}(t) R_{i}(t)I\! I_{i}(t)-\sum_{i=0}^{n-3}\omega_{i}(t)R_{i+1}(t)I\! I_{i+1}(t)
\\
= &\sum_{i=0}^{n-2}\omega_{i}(t)\bigl[R_{i}(t)I\! I_{i}(t)-R_{i+1}(t)I\! I_{i+1}(t)\bigr]-\omega_{0}(t)R_{0}(t)I\! I_{0}(t)+\omega_{n-2}(t)R_{n-1}(t)I\! I_{n-1}(t),
\end{split}
\end{equation}
where, since
\begin{equation*}
\bigl|I\! I_{i}(t)\bigr|=\frac{R_{i}}{\ell_{n}}v(R_{i})\bigl|\phi(x_{i})-\phi(x_{i+1})\bigr|\leq L'\quad \text{for} \ i\in \{0,\dots,n-1\},
\end{equation*}
then it holds
\begin{equation}
\label{127}
\bigl|-\omega_{0}(t)R_{0}(t)I\! I_{0}(t)\bigr|\leq L'\overline{R}e^{L't} \quad \text{and}\quad \bigl|\omega_{n-2}(t) R_{n-1}(t)I\!I_{n-1}(t)\bigr|\leq L'\overline{R}e^{L't}.
\end{equation}
Moreover, we can rewrite the sum in the right-hand side of \eqref{eq24} as
\begin{equation*}
\begin{split}
\sum_{i=0}^{n-2}\omega_{i}(t)\bigl[R_{i}(t)I\!I_{i}(t)\!-\!R_{i+1}(t)I\! I_{i+1}(t)\bigr]=\!\sum_{i=0}^{n-2}\omega_{i}(t)R_{i}(t)\bigl(I\! I_{i}(t)\!-\!I\! I_{i+1}(t)\bigr)+\!\sum_{i=0}^{n-2}\omega_{i}(t)I\! I_{i+1}(t)\bigl(R_{i}(t)-R_{i+1}(t)\bigr),
\end{split}
\end{equation*}
where, due to \eqref{987}, we immediately remark that
\begin{equation}
\label{128}
\sum_{i=0}^{n-2}\omega_{i}(t)I\! I_{i+1}(t)\bigl(R_{i}(t)-R_{i+1}(t)\bigr)\leq L'\sum_{i=0}^{n-2}\eta_{\sigma}\bigl(R_{i}(t)-R_{i+1}(t)\bigr)\leq L'\TV_{\sigma}[\rho^{n}(\cdot,t)].
\end{equation}
Turning to the remaining sum, from the second order Taylor's expansion of $\phi$ at $x_{i+1}(t)$, we get
\begin{equation*}
\phi(x_{i}(t))=\phi(x_{i+1}(t))-\phi'(x_{i+1}(t))\bigl(x_{i+1}(t)-x_{i}(t)\bigr)+\frac{\phi^{''}(\tilde{x}_{i,i+1}(t))}{2}\bigl(x_{i+1}(t)-x_{i}(t)\bigr)^{2}
\end{equation*}
and
\begin{equation*}
\begin{split}
\phi(x_{i+2}(t))=&\phi(x_{i+1}(t))+\phi'(x_{i+1}(t))\bigl(x_{i+2}(t)-x_{i+1}(t)\bigr)+\frac{\phi^{''}(\tilde{y}_{i+1,i+2}(t))}{2}\bigl(x_{i+2}(t)-x_{i+1}(t)\bigr)^{2}
\end{split}
\end{equation*}
for some $\tilde{x}_{i,i+1}(t)\in \bigl(x_{i}(t),x_{i+1}(t)\bigr)$ and $\tilde{y}_{i+1,i+2}(t)\in \bigl(x_{i+1}(t),x_{i+2}(t)\bigr)$. Hence we can rewrite the sum as
\begin{equation}
\label{129}
\begin{split}
\sum_{i=0}^{n-2}\omega_{i}(t)R_{i}(t)\bigl(I\! I_{i}(t)\!-\!I\! I_{i+1}(t)\bigr)= D_{1}(t)+D_{2}(t)+D_{3}(t),
\end{split}
\end{equation}
with
\begin{equation*}
D_{1}(t):= \sum_{i=0}^{n-2}\omega_{i}(t)R_{i}(t)\phi'(x_{i+1}(t))\bigl[v(R_{i+1}(t))-v(R_{i}(t))\bigr],
\end{equation*}
\begin{equation*}
D_{2}(t):=\sum_{i=0}^{n-2}\omega_{i}(t)\frac{R_{i}(t)^{2}}{2\ell_{n}}\phi''(\tilde{x}_{i,i+1}(t))v(R_{i}(t))\bigl(x_{i+1}(t)-x_{i}(t)\bigr)^{2}
\end{equation*}
and
\begin{equation*}
D_{3}(t):=\sum_{i=0}^{n-2}\omega_{i}(t)\frac{R_{i}(t)R_{i+1}(t)}{2\ell_{n}}\phi''(\tilde{y}_{i+1,i+2}(t))v(R_{i+1}(t))\bigl(x_{i+2}(t)-x_{i+1}(t)\bigr)^{2}.
\end{equation*}
We first notice, since $n\ell_{n}=1$, that $D_{2}(t)$ satisfies
\begin{equation}
\label{131}
D_{2}(t)=\frac{\ell_{n}}{2}\sum_{i=0}^{n-2}\omega_{i}(t)\phi''(\tilde{x}_{i,i+1}(t))v(R_{i}(t))\leq v_{\max}||\phi''||_{L^{\infty}(\mathbb{R})}.
\end{equation}
Concerning $D_{1}(t)$, from \eqref{53} and \eqref{987} it follows that
\begin{equation}
\label{132}
\begin{split}
D_{1}(t)&\leq \Lip_{v}||\phi'||_{L^{\infty}(\mathbb{R})}\overline{R}e^{L't}\sum_{i=0}^{n-2}\eta_{\sigma}\bigl(R_{i}(t)-R_{i+1}(t)\bigr)\leq \Lip_{v}||\phi'||_{L^{\infty}(\mathbb{R})}\overline{R}e^{L't}\TV_{\sigma}[\rho^n(\cdot,t)].
\end{split}
\end{equation}
Turning to the last sum $D_{3}(t)$, defining for brevity
\begin{equation*}
\gamma_{i}(t):=\frac{\phi''(\tilde{y}_{i,i+1}(t))}{2\ell_{n}}v(R_{i})\bigl(x_{i+1}(t)-x_{i}(t)\bigr)^{2}\quad \text{for}\ i\in \{1,\dots,n-1\},
\end{equation*}
we can rewrite $D_{3}(t)$ as
\begin{equation*}
\begin{split}
D_{3}(t)&=\sum_{i=0}^{n-2}\omega_{i}(t)R_{i+1}(t)^{2}\gamma_{i+1}(t)+\sum_{i=0}^{n-2}\omega_{i}(t)\bigl(R_{i}(t)-R_{i+1}(t)\bigr)R_{i+1}(t)\gamma_{i+1}(t):= D_{3}^{1}(t)+D_{3}^{2}(t),
\end{split}
\end{equation*}
where $D_{3}^{1}(t)$ satisfies the same inequality as $D_{2}(t)$, while \eqref{130} and \eqref{987} implies that
\begin{equation}
\label{133}
\begin{split}
D_{3}^{2}(t)&=\frac{1}{2}\sum_{i=0}^{n-2}\omega_{i}(t) \bigl(R_{i}(t)-R_{i+1}(t)\bigr)\phi''(\tilde{y}_{i+1,i+2}(t))v(R_{i+1}(t))\bigl(x_{i+2}(t)-x_{i+1}(t)\bigr)
\\
&\leq ||\phi''||_{L^{\infty}(\mathbb{R})}v_{\max}\bigl(\overline{x}_{\max}-\overline{x}_{\min}+2Lt\bigr)\sum_{i=0}^{n-2}\eta_{\sigma}\bigl(R_{i}(t)-R_{i+1}(t)\bigr)
\\
&\leq ||\phi''||_{L^{\infty}(\mathbb{R})}v_{\max}\bigl(\overline{x}_{\max}-\overline{x}_{\min}+2Lt\bigr)\TV_{\sigma}[\rho^n(\cdot,t)].
\end{split}
\end{equation}
Putting together \eqref{126}-\eqref{133}, we finally get the following differential inequality
\begin{equation}
\label{23}
\frac{d}{dt}\TV_{\sigma}[\rho^{n}(\cdot,t)]\leq  \frac{c_{0}}{\ell_{n}^2}\sigma+ c_{1}+c_{2}e^{L't}+\bigl(c_{3}+c_{4}t+c_{5}e^{L't}\bigr)\TV_{\sigma}[\rho^{n}(\cdot,t)] \quad \text{for all}\ t\geq 0,
\end{equation}
where $c_{i}$,  $i=0,\dots,5$, are six positive constants depending only on $\phi$, $v$ and $\overline{\rho}$.

In case (P2) we can proceed in a symmetric way and we get also here \eqref{23} (the details are omitted).

Turning to case (P3), we first remark that the bounds for $A(t)$ and $B(t)$ are the same already pointed out in cases (P2) and (P1) respectively. Regarding $C(t)$, we split this sum as
\begin{equation}
\label{24}
C(t)=\sum_{i=1}^{k_{n}-1}\mu_{i}(t)\dot{R}_{i}(t)+\sum_{i=k_{n}+1}^{n-2}\mu_{i}(t)\dot{R}_{i}(t)+\mu_{k_{n}}(t)\dot{R}_{k_{n}}(t):= C_{1}(t)+C_{2}(t)+C_{3}(t),
\end{equation}
where $C_{1}(t)$ and $C_{2}(t)$ can be estimated, rearranging  the indices properly, using the same reasoning as in cases (P2) and (P1) respectively, while for $C_{3}(t)$, which is given by
\begin{equation*}
C_{3}(t)=\bigl[\eta_{\sigma}'\bigl(R_{k_n}-R_{k_n+1}\bigr)-\eta_{\sigma}'\bigl(R_{k_n-1}-R_{k_n}\bigr)\bigr]\frac{R_{k_n}^2}{\ell_{n}}\bigl[v(R_{k_n-1})\phi(x_{k_n})-v(R_{k_n+1})\phi(x_{k_n+1})\bigr],
\end{equation*}
we now show that
\begin{equation*}
C_{3}(t)\leq  2L' \overline{R} e^{L't}+ 4\frac{\overline{R}^2\Lip_{v}||\phi||_{L^{\infty}(\mathbb{R})}}{\ell_{n}}\sigma.
\end{equation*}
We need to distinguish the following cases:
\begin{itemize}
\item[(a)] If $R_{k_{n}}-R_{k_{n}+1}\geq \sigma$ and $R_{k_{n}-1}-R_{k_{n}}\geq \sigma$ or if $R_{k_{n}}-R_{k_{n}+1}\leq -\sigma$ and $R_{k_{n}-1}-R_{k_{n}}\leq -\sigma$, then $\mu_{k_{n}}(t)=0$ and hence $C_{3}(t)=0$.
\item[(b)] If $R_{k_{n}}-R_{k_{n}+1}\geq \sigma$ and $R_{k_{n}-1}-R_{k_{n}}\leq -\sigma$, then $\mu_{k_{n}}(t)=2$ and so, since $\phi(x_{k_{n}})\leq 0$ and $\phi(x_{k_{n}+1})\geq 0$, it holds that $C_{3}(t)\leq 0$.
\item[(c)] If $R_{k_{n}}-R_{k_{n}+1}\leq -\sigma$ and $R_{k_{n}-1}-R_{k_{n}}\geq \sigma$, then it follows $\mu_{k_{n}}(t)=-2$, $v(R_{k_{n}-1})\leq v(R_{k_{n}})$, $v(R_{k_{n}+1})\leq v(R_{k_{n}})$ and therefore, since $\phi(x_{k_{n}})\leq 0$ and $\phi(x_{k_{n}+1})\geq 0$, we get
\begin{equation*}
\begin{split}
C_{3}(t)\leq 2\frac{R_{k_{n}}^{2}}{\ell_{n}}v(R_{k_{n}})\bigl[\phi(x_{k_{n}+1})-\phi(x_{k_{n}})\bigr]
\leq 2L'\overline{R}e^{L't}.
\end{split}
\end{equation*}
\item[(d)] If $R_{k_{n}}-R_{k_{n}+1}\geq \sigma$ and $|R_{k_{n}-1}-R_{k_{n}}|< \sigma$, then we have $0<\mu_{k_{n}}(t)<2$, $v(R_{k_{n}+1})\geq v(R_{k_{n}})$ and so, since $\phi(x_{k_{n}+1})\geq 0$, we get
\begin{equation*}
\begin{split}
C_{3}(t)& \leq \mu_{k_n}(t)\frac{R_{k_{n}}^{2}}{\ell_{n}}\bigl[v(R_{k_{n}-1})\phi(x_{k_{n}})-v(R_{k_{n}})\phi(x_{k_{n}+1})\bigr]
\\
&=\mu_{k_n}(t)\frac{R_{k_{n}}^{2}}{\ell_{n}}v(R_{k_{n}-1})\bigl[\phi(x_{k_{n}})-\phi(x_{k_{n}+1})\bigr]+\mu_{k_n}(t)\frac{R_{k_{n}}^{2}}{\ell_{n}}\phi(x_{k_{n}+1})\bigl[v(R_{k_{n}-1})-v(R_{k_{n}})\bigr]
\\
&\leq 2L' \overline{R} e^{L't}+ 2 \frac{\overline{R}^2\Lip_{v}||\phi||_{L^{\infty}(\mathbb{R})}}{\ell_{n}}\sigma.
\end{split}
\end{equation*}
\item[(e)] If $R_{k_{n}}-R_{k_{n}+1}\leq -\sigma$ and $|R_{k_{n}-1}-R_{k_{n}}|< \sigma$, then it follows $-2<\mu_{k_{n}}(t)<0$, $v(R_{k_{n}+1})\leq v(R_{k_{n}})$ and hence, since $\phi(x_{k_{n}+1})\geq 0$, it holds the same estimate as in the previous case.
\item[(f)] If $|R_{k_{n}}-R_{k_{n}+1}|< \sigma$ and $R_{k_{n}-1}-R_{k_{n}}\leq - \sigma$, then it follows $0<\mu_{k_{n}}(t)<2$, $v(R_{k_{n}-1})\geq v(R_{k_{n}})$ and therefore, since $\phi(x_{k_{n}})\leq 0$, we get
\begin{equation*}
\begin{split}
C_{3}(t)& \leq \mu_{k_n}(t)\frac{R_{k_{n}}^{2}}{\ell_{n}}\bigl[v(R_{k_{n}})\phi(x_{k_{n}})-v(R_{k_{n}+1})\phi(x_{k_{n}+1})\bigr]
\\
&=\mu_{k_n}(t)\frac{R_{k_{n}}^{2}}{\ell_{n}}v(R_{k_{n}})\bigl[\phi(x_{k_{n}})-\phi(x_{k_{n}+1})\bigr]+\mu_{k_n}(t)\frac{R_{k_{n}}^{2}}{\ell_{n}}\phi(x_{k_{n}+1})\bigl[v(R_{k_{n}})-v(R_{k_{n}+1})\bigr]
\\
&\leq 2L' \overline{R} e^{L't}+ 2 \frac{\overline{R}^2\Lip_{v}||\phi||_{L^{\infty}(\mathbb{R})}}{\ell_{n}}\sigma.
\end{split}
\end{equation*}
\item[(g)] If $|R_{k_{n}}-R_{k_{n}+1}|< \sigma$ and $R_{k_{n}-1}-R_{k_{n}}\geq \sigma$, then we have $-2<\mu_{k_{n}}(t)<0$, $v(R_{k_{n}-1})\leq v(R_{k_{n}})$ and hence, since $\phi(x_{k_{n}})\leq 0$, it holds the same estimate as in the previous case.
\item[(h)] If $|R_{k_{n}}-R_{k_{n}+1}|< \sigma$ and $|R_{k_{n}-1}-R_{k_{n}}|< \sigma$, then it follows $-2<\mu_{k_{n}}(t)<2$ and therefore we get
\begin{equation*}
\begin{split}
C_{3}(t)&=\mu_{k_n}(t)\frac{R_{k_{n}}^{2}}{\ell_{n}}v(R_{k_{n}-1})\bigl[\phi(x_{k_{n}})-\phi(x_{k_{n}+1})\bigr]+\mu_{k_n}(t)\frac{R_{k_{n}}^{2}}{\ell_{n}}\phi(x_{k_{n}+1})\bigl[v(R_{k_{n}-1})-v(R_{k_{n}+1})\bigr]
\\
&\leq 2L' \overline{R} e^{L't}+ 2 \frac{\overline{R}^2\Lip_{v}||\phi||_{L^{\infty}(\mathbb{R})}}{\ell_{n}}|R_{k_n-1}-R_{k_n+1}|\leq 2L' \overline{R} e^{L't}+ 4 \frac{\overline{R}^2\Lip_{v}||\phi||_{L^{\infty}(\mathbb{R})}}{\ell_{n}}\sigma.
\end{split}
\end{equation*}
\end{itemize}
As a consequence, we hence have that \eqref{23} is valid also in case (P3).

In the last case (P4), we can estimate $A(t)$ and $B(t)$ in the same way as in cases (P1) and (P2) respectively, substituting $\overline{R}e^{L't}$ with $R_{\max}$ whenever it appears. Concerning $C(t)$, we can rewrite this sum using the same splitting \eqref{24} seen in case (P3), where $C_{3}(t)$ now satisfies $C_{3}(t)\leq 2L'R_{\max}$, since
\begin{equation*}
C_{3}(t)=\mu_{k_n}(t)\frac{R_{k_n}^2}{\ell_{n}}\bigl[\dot{x}_{k_n}-\dot{x}_{k_{n}+1}\bigr]=\mu_{k_n}(t)\frac{R_{k_n}^2}{\ell_{n}}v(R_{k_n})\bigl[\phi(x_{k_n})-\phi(x_{k_{n}+1})\bigr]\leq  2L'R_{\max}.
\end{equation*}
For $C_{1}(t)$ and $C_{2}(t)$, rearranging the indices properly and substituting again $\overline{R}e^{L't}$ with $R_{\max}$, we can follow the same reasoning as in cases (P1) and (P2) respectively: in this way, we  obtain all the previous estimates with a slight difference only on the term $D_{3}^{2}(t)$. For this sum it holds, due to remark \ref{rem1} and \eqref{987}, that
\begin{equation*}
\begin{split}
D_{3}^{2}(t)&=\frac{1}{2}\sum_{i=0}^{n-2}\omega_{i}(t)\bigl(R_{i}(t)-R_{i+1}(t)\bigr)\phi''(\tilde{y}_{i+1,i+2}(t))v(R_{i+1}(t))\bigl(x_{i+2}(t)-x_{i+1}(t)\bigr)
\\
&\leq ||\phi''||_{L^{\infty}(\mathbb{R})}v_{\max}(\overline{x}_{\max}-\overline{x}_{\min})\sum_{i=0}^{n-2}\eta_{\sigma}\bigl(R_{i}(t)-R_{i+1}(t)\bigr)
\\
&\leq ||\phi''||_{L^{\infty}(\mathbb{R})}v_{\max}(\overline{x}_{\max}-\overline{x}_{\min})\TV_{\sigma}[\rho^n(\cdot,t)].
\end{split}
\end{equation*}

Applying now Gronwall lemma to \eqref{23} and letting $\alpha\rightarrow 0$, we get
\begin{equation}
\label{26}
\begin{split}
\TV[\rho^{n}(\cdot,t)]\leq& \TV[\rho^{n}(\cdot,0)]e^{\int_{0}^{t}(c_{3}+c_{4}\tau+c_{5}e^{L'\tau})d\tau}+\int_{0}^{t}(c_{1}+c_{2}e^{L's})e^{\int_{s}^{t}(c_{3}+c_{4}\tau+c_{5}e^{L'\tau})d\tau}ds.
\end{split}
\end{equation}
Concerning the total variation of $\rho^{n}(\cdot,0)$, since \eqref{5} and the mean value theorem imply
\begin{equation*}
R_{i}(0)=\frac{\ell_{n}}{\overline{x}_{i+1}-\overline{x}_{i}}= \fint_{\overline{x}_{i}}^{\overline{x}_{i+1}}\overline{\rho}(x)dx=\overline{\rho}(z_{i}) \quad \text{for some} \ z_{i}\in (\overline{x}_{i},\overline{x}_{i+1}),
\end{equation*}
then it holds that
\begin{equation*}
\begin{split}
\TV[\rho^{n}(\cdot,0)]&=R_{0}(0)+R_{n-1}(0)+\sum_{i=0}^{n-2}|R_{i}(0)-R_{i+1}(0)|=\overline{\rho}(z_{0})+\overline{\rho}(z_{n-1})+\sum_{i=0}^{n-2}|\overline{\rho}(z_{i})-\overline{\rho}(z_{i+1})|\leq \TV[\overline{\rho}].
\end{split}
\end{equation*}
Moreover, defining
\begin{equation*}
\alpha:=c_{1}, \quad \beta:= \frac{c_{2}}{\ell_{n}}, \quad \gamma:= \max\biggl(c_{3},\frac{c_4}{2}\biggr) \quad \text{and} \quad \zeta:= \frac{c_{5}}{\ell_n},
\end{equation*}
after a simple calculation it follows that
\begin{equation*}
e^{\int_{0}^{t}(c_{3}+c_{4}\tau+c_{5}e^{L'\tau})d\tau}\leq e^{[\gamma t(1+t)+\zeta e^{L't}]}
\end{equation*}
and
\begin{equation*}
\int_{0}^{t}(c_{1}+c_{2}e^{L's})e^{\int_{s}^{t}(c_{3}+c_{4}\tau+c_{5}e^{L'\tau})d\tau}ds\leq \bigl(\alpha t+\beta e^{L't}\bigr) e^{[\gamma t(1+t)+\zeta e^{L't}]},
\end{equation*}
therefore, combining together the previous three estimates in \eqref{26}, we finally get \eqref{11}.
\end{proof}

\begin{rem}
The previous proposition gives us the needed compactness of the sequence $\displaystyle\{\rho^{n}\}_{n\in \mathbb{N}}$ with respect to space. Concerning the time variable, we are not able to obtain a $L^{1}$ uniform continuity estimate which would provide a sufficient control on the time oscillation. Nevertheless, we are going to prove a uniform time continuity estimate with respect to the $1$-Wasserstein distance which is sufficient to get the required strong $L^{1}$ compactness with respect to space and time.
\end{rem}

We now recall the main properties on the one-dimensional $1$-Wasserstein metric.
Let $\mu$ be a probability measure on $\mathbb{R}$ with finite first moment, and let us denote with $X_{\mu}$ the pseudo-inverse of its cumulative distribution function, that is
\begin{equation*}
X_{\mu}(z):=\inf\{x\in \mathbb{R}:\mu((-\infty,x])>z\} \quad \text{for} \ z\in [0,1].
\end{equation*}
We first notice that $X_\mu \in L^1([0,1])$ (see \cite{villani}). Moreover, the $1$-Wasserstein distance between two probability measures $\mu$ and $\nu$ on $\mathbb{R}$ can be defined as the $L^{1}$ distance of $X_{\mu}$ and $X_{\nu}$, that is
\begin{equation}
\label{17}
W_{1}(\mu,\nu):=||X_{\mu}-X_{\nu}||_{L^{1}([0,1])}.
\end{equation}
In particular, starting from the definition \eqref{10} of $\rho^{n}$, we can explicitly compute the pseudo-inverse function $X_{\rho^{n}}$ and we get
\begin{equation}
\label{18}
X_{\rho^{n}(\cdot,t)}(z)=\sum_{i=0}^{n-1}\bigl[x_{i}(t)+(z-i\ell_{n})R_{i}(t)^{-1}\bigr]\mathbbm{1}_{[i\ell_{n}, (i+1)\ell_{n})}(z).
\end{equation}
After this short preamble, we can prove the uniform time continuity estimate with respect to the $1$-Wasserstein distance, which is stated in the following

\begin{prop}
\label{prop3}
Assume (V), (I) and (P) are satisfied and, moreover, assume (V$^*$) is satisfied in case (P4). If one of (P1), (P2), (P3) or (P4) holds, then there exists a constant $c$, dependent only on $v$ and $\phi$, such that
\begin{equation}
\label{16}
W_{1}\bigl(\rho^{n}(\cdot,t),\rho^{n}(\cdot,s)\bigr)\leq c|t-s| \quad \text{for all}\ t,s> 0.
\end{equation}
\end{prop}
\begin{proof}
Let $0<s<t$ fixed. From \eqref{17}, \eqref{18} and the triangular inequality, it follows that
\begin{equation*}
\begin{split}
W_{1}\bigl(\rho^{n}(\cdot,t),\rho^{n}(\cdot,s)\bigr)= \sum_{i=0}^{n-1}\int_{i\ell_{n}}^{(i+1)\ell_{n}}\bigl|x_{i}(t)-x_{i}(s)+(z-i\ell_{n})\bigl(R_{i}(t)^{-1}-R_{i}(s)^{-1}\bigr)\bigr|dz
\leq   A(s,t)+B(s,t),
\end{split}
\end{equation*}
with
\begin{equation*}
A(s,t):=\sum_{i=0}^{n-1}\int_{i\ell_{n}}^{(i+1)\ell_{n}}|x_{i}(t)-x_{i}(s)|dz \quad \text{and} \quad B(s,t):=\sum_{i=0}^{n-1}\int_{i\ell_{n}}^{(i+1)\ell_{n}}(z-i\ell_{n})\bigl|R_{i}(t)^{-1}-R_{i}(s)^{-1}\bigr|dz.
\end{equation*}
Now we should estimate $A(s,t)$ and $B(s,t)$ separately. Due to \eqref{190}, we have that
\begin{equation*}
A(s,t)=\ell_{n}\sum_{i=0}^{n-1}|x_{i}(t)-x_{i}(s)|=\ell_{n}\sum_{i=0}^{n-1}\biggr|\int_{s}^{t}\dot{x}_{i}(\tau)d\tau\biggr|\leq \ell_{n}\sum_{i=0}^{n-1}\int_{s}^{t}|\dot{x}_{i}(\tau)|d\tau\leq L(t-s),
\end{equation*}
while, turning to $B(s,t)$, we first notice that
\begin{equation*}
\begin{split}
B(s,t)=\sum_{i=0}^{n-1}\bigl|R_{i}(t)^{-1}-R_{i}(s)^{-1}\bigr|\int_{i\ell_{n}}^{(i+1)\ell_{n}}(z-i\ell_{n})dz&=\frac{\ell_{n}^{2}}{2}\sum_{i=0}^{n-1}\biggl|\int_{s}^{t}\frac{d}{d\tau} R_{i}(\tau)^{-1}d\tau\biggr|\leq\frac{\ell_{n}^{2}}{2}\sum_{i=0}^{n-1}\int_{s}^{t}\frac{\bigl|\dot{R}_{i}(\tau)\bigr|}{R_{i}(\tau)^{2}}d\tau.
\end{split}
\end{equation*}
For clarity we  should treat the four cases separately from now on. In case (P1), substituting the definition of $\dot{R}_{i}(t)$ we get
\begin{equation*}
\begin{split}
B(s,t)\leq & \ \frac{\ell_{n}}{2} \sum_{i=0}^{n-2}\int_{s}^{t}\bigl|v({R}_{i+1}(\tau))\phi(x_{i+1}(\tau))-v({R}_{i}(\tau))\phi(x_{i}(\tau))\bigr|d\tau
\\
&+\frac{\ell_{n}}{2}\int_{s}^{t}\bigl|v_{\max}\phi(x_{n}(\tau))-v({R}_{n-1}(\tau))\phi(x_{n-1}(\tau))\bigr|d\tau
\end{split}
\end{equation*}
where, for all $\tau\geq0$ and $i\in \{0,\dots,n-2\}$, it holds that
\begin{equation*}
\bigl|v({R}_{i+1}(\tau))\phi(x_{i+1}(\tau))-v({R}_{i}(\tau))\phi(x_{i}(\tau))\bigr|\leq\max\{v({R}_{i+1}(\tau))\phi(x_{i+1}(\tau)),v({R}_{i}(\tau))\phi(x_{i}(\tau))\}\leq L
\end{equation*}
and analogously that
\begin{equation*}
\bigl|v_{\max}\phi(x_{n}(\tau))-v({R}_{n-1}(\tau))\phi(x_{n-1}(\tau))\bigr|\leq L.
\end{equation*}
This implies that $\displaystyle B(s,t) \leq \frac{L}{2} (t-s)$ and hence that $\displaystyle W_{1}\bigl(\rho^{n}(\cdot,t),\rho^{n}(\cdot,s)\bigr)\leq \frac{3}{2} L(t-s)$,
which concludes the proof of \eqref{16} in case (P1), since the calculation is still valid interchanging $s$ and $t$.

In case (P2) we can reason in a symmetric way with respect to the previous case and we get the same estimate for $B(s,t)$. The details are left to the reader.

Turning to case (P3), we remark that $\dot{R}_{i}(\tau)$ has the same expression of cases (P2) and (P1) for $i\in\{0,\dots,k_{n}-1\}$ and $i\in\{k_{n}+1,\dots,n-1\}$ respectively, while
\begin{equation*}
\dot{R}_{k_{n}}(\tau)=-\frac{{R}_{k_{n}}(\tau)^{2}}{\ell_{n}}\bigl[v({R}_{k_{n}+1}(\tau))\phi(x_{k_{n}+1}(\tau))-v({R}_{k_{n}-1}(\tau))\phi(x_{k_{n}}(\tau))\bigr].
\end{equation*}
Since
\begin{equation*}
\bigl|v({R}_{k_{n}+1}(\tau))\phi(x_{k_{n}+1}(\tau))-v({R}_{k_{n}-1}(\tau))\phi(x_{k_{n}}(\tau))\bigr|\leq 2L,
\end{equation*}
then it follows that $B(s,t)\leq L(t-s)$ and hence also in this case \eqref{16} holds.

In the remaining case (P4), we have that $\dot{R}_{i}(\tau)$ has the same expression of cases (P1) and (P2) for $i\in\{0,\dots,k_{n}-1\}$ and $i\in\{k_{n}+1,\dots,n-1\}$ respectively, while
\begin{equation*}
\dot{R}_{k_{n}}(\tau)=-\frac{{R}_{k_{n}}(\tau)^{2}}{\ell_{n}}v({R}_{k_{n}}(\tau))\bigl[\phi(x_{k_{n}+1}(\tau))-\phi(x_{k_{n}}(\tau))\bigr].
\end{equation*}
Since
\begin{equation*}
v({R}_{k_{n}}(\tau))\bigl|\phi(x_{k_{n}+1}(\tau))-\phi(x_{k_{n}}(\tau))\bigr|\leq 2L,
\end{equation*}
then $B(s,t)$ satisfies the same inequality seen in case (P3) and hence the validity of \eqref{16} is proved also in case (P4).
\end{proof}

Before passing to the main result of this paper, we recall a generalised version of Aubin-Lions lemma (see \cite{RosSav},\cite{DFR}, \cite{DFRF}) which has a key role in the sequel.
\begin{thm}
\label{teo2}
Let $T>0$ fixed, $I\subset \mathbb{R}$ a bounded open interval (possibly depending on $T$), $\displaystyle\{\mu^{n}\}_{n\in\mathbb{N}}$ a sequence in $L^{\infty}((0,T);L^{1}(\mathbb{R}))$ such that $\mu^{n}(\cdot,t)\geq 0$ and $||\mu^{n}(\cdot,t)||_{L^{1}(\mathbb{R})}=1$ for all $n\in \mathbb{N}$ and $t\in [0,T]$. If
\begin{itemize}
\item[(A)] $\supp [\mu^{n}(\cdot,t)]\subseteq I$ for all $n\in \mathbb{N}$ and $t\in [0,T]$,
\item[(B)] $\displaystyle \sup_{n\in \mathbb{N}}\int_{0}^{T}\biggl[||\mu^{n}(\cdot,t)||_{L^{1}(I)}+\TV[\mu^{n}(\cdot,t);I]\biggr]dt<\infty$,
\item[(C)] There exists a constant $c$ independent on $n$ such that $W_{1}\bigl(\mu^{n}(\cdot,t),\mu^{n}(\cdot,s)\bigr)\leq c|t-s|$ for all $s,t\in (0,T)$,
\end{itemize}
then $\{\mu^{n}\}_{n\in\mathbb{N}}$ is strongly relatively compact in $L^{1}\bigl(\mathbb{R}\times [0,T]\bigr)$.
\end{thm}

\section{Proof of the main result}\label{sec:consistency}

Now we are ready to prove the main result of the paper. For clarity, in the sequel we drop the time dependence whenever it is clear from the context. We recall our notation $f(\rho)=\rho v(\rho)$. Moreover, we recall that by weak solution to \eqref{sis2} we mean a distributional solution to the PDE in \eqref{sis2} in $L^\infty(\R\times [0,+\infty))$. 

We now introduce our concept of entropy solutions for \eqref{sis2}, which is an adapted version of the one in \cite[Definition 2.1]{KarlsenRisebro}
\begin{definition}\label{definition_entropy_solution}
  Let $\overline{\rho}\in BV(\R)$. A function $\rho:L^\infty([0,+\infty)\,;\,BV(\R))$ is an \emph{entropy solution} to \eqref{sis2} if
  \begin{equation}
\label{def_entropy}
\begin{split}
\int_{0}^T \int_{\mathbb{R}}\biggl[&|\rho(x,t)-k|\varphi_{t}(x,t)+\sign(\rho(x,t)-k)\bigl(f(\rho(x,t))-f(k)\bigr)\phi(x)\varphi_{x}(x,t)
\\
&-\sign(\rho(x,t)-k)f(k)\phi'(x)\varphi(x,t)\biggr]dxdt\geq 0,
\end{split}
\end{equation}
for all $k \geq 0$ and for all $\varphi\in C_c(\R\times (0,+\infty))$ and if $\rho(\cdot,t)\rightarrow \overline{\rho}$ strongly in $L^1$ as $t\searrow 0$.
\end{definition}

We now recall an adapted version of the $L^{1}$ contraction property proved by Karlsen and Risebro in \cite[Theorem 1.2]{KarlsenRisebro}, which will be crucial in the proof of the uniqueness of the entropy solution to \eqref{sis2}.
\begin{thm}
\label{thmKR}
Let $T>0$ fixed arbitrarily, let $f$ a locally Lipschitz function on $\mathbb{R}$, let $\psi\in W_{loc}^{1,1}(\mathbb{R})\cap C(\mathbb{R})$ such that $\psi, \psi' \in L^{\infty}(\mathbb{R})$ and consider the problem
\begin{equation}
\label{295}
\begin{cases}
     w_{t} + \bigl(f(w) \psi(x)\bigr)_{x} =0,
    \qquad  & x\in \mathbb{R},\, t\in (0,T),
    \\ w(x,0)=\overline{w}(x),
    \qquad  & x\in \mathbb{R}.
   \end{cases}
\end{equation}
If $u,v\in L^{\infty}((0,T);BV(\mathbb{R}))$ are two entropy solutions of \eqref{295} with respective initial datum $u_{0}$ and $v_{0}$ in the sense of Definition 2.1 in \cite{KarlsenRisebro}, both in $L^{1}(\mathbb{R})\cap L^{\infty}(\mathbb{R})\cap BV(\mathbb{R})$, then, for almost every $t\in (0,T)$, it holds
\begin{equation*}
||u(\cdot,t)-v(\cdot,t)||_{L^{1}(\mathbb{R})}\leq ||u_{0}-v_{0}||_{L^{1}(\mathbb{R})}.
\end{equation*}
In particular, this implies the existence of at most one entropy solution of \eqref{295}.
\end{thm}

We now state out main result.
\begin{thm}\label{thm:main}
Let $T>0$ fixed arbitrarily. Assume (V), (I) and (P) are satisfied and, moreover, assume (V$^*$) is satisfied in case (P4). If one of (P1), (P2), (P3) or (P4) holds, then the approximated density $\{\rho^{n}\}_{n\in\mathbb{N}}$ defined in \eqref{10} converges, up to a subsequence, almost everywhere and in $L^{1}$ on $\mathbb{R}\times [0,T]$ to the unique entropy solution to the Cauchy problem \eqref{sis2} in the sense of Definition \ref{definition_entropy_solution}.
\end{thm}
\begin{proof}
We first show that $\{\rho^{n}\}_{n\in\mathbb{N}}$ converges, up to a subsequence, almost everywhere and in $L^{1}$ on $\mathbb{R}\times [0,T]$. We notice that the support of $\rho^{n}(\cdot,t)$ is contained, for every $n\in \mathbb{N}$ and $t\in [0,T]$, in the closed interval
\begin{equation*}
J:=[a,b]=\begin{cases} \bigl[\overline{x}_{\min},\overline{x}_{\max}+LT\bigr] & \text{in case (P1)},
\\ \bigl[\overline{x}_{\min}-LT,\overline{x}_{\max}\bigr] & \text{in case (P2)},
\\ \bigl[\overline{x}_{\min}-LT,\overline{x}_{\max}+LT\bigr] & \text{in case (P3)},
\\ \bigl[\overline{x}_{\min},\overline{x}_{\max}\bigr] & \text{in case (P4)}.
\end{cases}
\end{equation*}
Therefore, taking as $I$ any open interval of the type $(a-c,b+d)$ with $c$ and $d$ arbitrary positive constants, we can apply theorem \ref{teo2}, indeed assumption (A) is valid due to the above construction, while assumptions (B) and (C) are a direct consequence of propositions \ref{prop2} and \ref{prop3} respectively. As a result, it follows that $\{\rho^{n}\}_{n\in\mathbb{N}}$ converges, up to a subsequence that we still denote in the sequel with $\{\rho^{n}\}_{n\in\mathbb{N}}$, almost everywhere and in $L^{1}$ on $\mathbb{R}\times[0,T]$ to a certain function $\rho$.

Now we show that $\{\rho^{n}\}_{n\in\mathbb{N}}$ satisfies, for every $k\geq0$ and every non-negative $\varphi\in C_{c}^{\infty}(\mathbb{R}\times (0,T))$,
\begin{equation}
\label{48}
\begin{split}
\liminf_{n\rightarrow + \infty}\int_{0}^T \int_{\mathbb{R}}\biggl[&|\rho^{n}(x,t)-k|\varphi_{t}(x,t)+\sign(\rho^{n}(x,t)-k)\bigl(f(\rho^{n}(x,t))-f(k)\bigr)\phi(x)\varphi_{x}(x,t)
\\
&-\sign(\rho^{n}(x,t)-k)f(k)\phi'(x)\varphi(x,t)\biggr]dxdt\geq 0,
\end{split}
\end{equation}
where we denote $f(\eta):=\eta v(\eta)$. Let us omit from now on also the $x$ dependence whenever it is clear from the context. We first remark that, since $\supp[\varphi]$ is compact in $\mathbb{R} \times (0,T)$, then it holds
\begin{equation}
\label{140}
\begin{split}
\int_{0}^T\int_{\mathbb{R}}\bigl[|\rho^{n}-k|\varphi_{t}+\sign(\rho^{n}-k)(f(\rho)-f(k))\phi\varphi_{x}-\sign(\rho^{n}-k)f(k)\phi'\varphi\bigr]dxdt =& A+B+C,
\end{split}
\end{equation}
where
\begin{equation*}
A:= \sum_{i=0}^{n-1}\int_{0}^{T}\int_{x_{i}}^{x_{i+1}}\bigl[|R_{i}-k|\varphi_{t}+\sign(R_{i}-k)(f(R_{i})-f(k))\phi\varphi_{x}-\sign(R_{i}-k)f(k)\phi'\varphi\bigr]dxdt,
\end{equation*}
\begin{equation*}
B:= \int_{0}^{T}\int_{-\infty}^{x_{0}}\bigl[k\varphi_{t}+f(k)\phi\varphi_{x}+f(k)\phi'\varphi\bigr]dxdt \quad \text{and} \quad
C:=\int_{0}^{T}\int_{x_{n}}^{+\infty}\bigl[k\varphi_{t}+f(k)\phi\varphi_{x}+f(k)\phi'\varphi\bigr]dxdt.
\end{equation*}
Recalling that $\displaystyle\frac{d}{dt}\int_{\alpha(t)}^{\beta(t)} g(x,t)dx = \int_{\alpha(t)}^{\beta(t)} g_t(x,t) dx + g(\beta(t),t)\dot{\beta}(t) - g(\alpha(t),t)\dot{\alpha}(t)$, it follows that
\begin{equation}
\begin{split}
B=k\int_{0}^{T}\biggl(\int_{-\infty}^{x_{0}}\varphi_{t}dx\biggr)dt+kv(k)\int_{0}^{T}\int_{-\infty}^{x_{0}}(\phi\varphi)_{x}dxdt = k\int_{0}^{T}\bigl(v(k)\phi(x_{0})-\dot{x}_{0}\bigr)\varphi(x_{0})dt
\end{split}
\end{equation}
and analogously that
\begin{equation}
C=k\int_{0}^{T}\biggl(\int_{x_{n}}^{+\infty}\varphi_{t}dx\biggr)dt+kv(k)\int_{0}^{T}\int_{x_{n}}^{+\infty}(\phi\varphi)_{x}dxdt= k\int_{0}^{T}\bigl(\dot{x}_{n}-v(k)\phi(x_{n})\bigr)\varphi(x_{n})dt,
\end{equation}
while we can rewrite $A$ as
\begin{equation}
\!\!A= \sum_{i=0}^{n-1}\int_{0}^{T}\!\!\!\int_{x_{i}}^{x_{i+1}}\!\!\bigl[|R_{i}-k|\varphi_{t}+\sign(R_{i}-k)f(R_{i})\phi\varphi_{x}-\sign(R_{i}-k)f(k)(\phi\varphi)_{x}\bigr]dxdt = A_{1}+A_{2}+A_{3},
\end{equation}
where
\begin{equation*}
A_{1}:= \sum_{i=0}^{n-1}\int_{0}^{T}|R_{i}-k|\biggl(\int_{x_{i}}^{x_{i+1}}\varphi_{t}dx\biggr)dt, \quad A_{2}:= \sum_{i=0}^{n-1}\int_{0}^{T}\sign(R_{i}-k)f(R_{i})\biggl(\int_{x_{i}}^{x_{i+1}}\phi\varphi_{x}dx\biggr)dt
\end{equation*}
and
\begin{equation*}
A_{3}:= - \sum_{i=0}^{n-1}\int_{0}^{T}\sign(R_{i}-k)f(k)\bigl[\phi(x_{i+1})\varphi(x_{i+1})-\phi(x_{i})\varphi(x_{i})\bigr]dt.
\end{equation*}
Integrating by parts and since $\supp[\varphi(x,\cdot)]\subseteq(0,T)$ for every $x\in\mathbb{R}$, we get that $A_{1}$ satisfies
\begin{equation}
\label{141}
\begin{split}
A_{1}=&\sum_{i=0}^{n-1}\int_{0}^{T}|R_{i}-k|\biggl(\frac{d}{dt}\int_{x_{i}}^{x_{i+1}}\varphi dx\biggr)dt-\sum_{i=0}^{n-1}\int_{0}^{T}|R_{i}-k|\dot{x}_{i+1}\varphi(x_{i+1})dt+ \sum_{i=0}^{n-1}\int_{0}^{T}|R_{i}-k|\dot{x}_{i}\varphi(x_{i})dt
\\
= & -\sum_{i=0}^{n-1}\int_{0}^{T}\sign(R_{i}-k)(R_{i}-k)\dot{x}_{i+1}\varphi(x_{i+1})dt -\sum_{i=0}^{n-1}\int_{0}^{T}\sign(R_{i}-k)\dot{R}_{i}\biggl(\int_{x_{i}}^{x_{i+1}}\varphi dx\biggr)dt
\\
& +\sum_{i=0}^{n-1}\int_{0}^{T}\sign(R_{i}-k)(R_{i}-k)\dot{x}_{i}\varphi(x_{i})dt,
\end{split}
\end{equation}
while $A_{2}$ has a different expression in the four cases, since we need to approximate the function $\phi$ differently according to its sign. In particular we have:
\begin{itemize}
\item[(P1)] For case (P1), using the first order Taylor's expansion  of $\phi$ at $x_{i}$ in the interval $(x_{i},x_{i+1})$, which is given, for all $x\in (x_{i},x_{i+1})$, by
\begin{equation}
\label{28}
\phi(x)=\phi(x_{i})+\phi'(\tilde{x}_{i,i+1})(x-x_{i}) \quad \text{for some}\ \ \tilde{x}_{i,i+1} \in (x_{i},x_{i+1}),
\end{equation}
we can rewrite $A_{2}$ as
\begin{equation}
\label{142}
\begin{split}
A_{2}=& \sum_{i=0}^{n-1}\int_{0}^{T}\sign(R_{i}-k)R_{i}v(R_{i})\phi(x_{i})\bigl(\varphi(x_{i+1})-\varphi(x_{i})\bigr)dt
\\
&+ \sum_{i=0}^{n-1}\int_{0}^{T}\sign(R_{i}-k)R_{i}v(R_{i})\biggl(\int_{x_{i}}^{x_{i+1}}\phi'(\tilde{x}_{i,i+1})(x-x_{i})\varphi_{x}dx\biggr)dt.
\end{split}
\end{equation}
\item[(P2)] In case (P2), we use instead the first order Taylor's expansions of $\phi$ at $x_{i+i}$, that is
\begin{equation}
\label{29}
\phi(x)=\phi(x_{i+1})+\phi'(\tilde{y}_{i,i+1})(x-x_{i+1}) \quad \text{for some}\ \tilde{y}_{i,i+1} \in (x_{i},x_{i+1})
\end{equation}
and in this way we get that
\begin{equation*}
\begin{split}
A_{2}=& \sum_{i=0}^{n-1}\int_{0}^{T}\sign(R_{i}-k)R_{i}v(R_{i})\phi(x_{i+1})\bigl(\varphi(x_{i+1})-\varphi(x_{i})\bigr)dt
\\
&+ \sum_{i=0}^{n-1}\int_{0}^{T}\sign(R_{i}-k)R_{i}v(R_{i})\biggl(\int_{x_{i}}^{x_{i+1}}\phi'(\tilde{y}_{i,i+1})(x-x_{i+1})\varphi_{x}dx\biggr)dt.
\end{split}
\end{equation*}
\item[(P3)] For case (P3), we use the Taylor's expansions \eqref{28} and \eqref{29} of $\phi$ respectively for $i\in\{k_{n}+1,\dots,n-1\}$ and $i\in\{0,\dots,k_{n}-1\}$, in order to have that
\begin{equation}
\label{144}
\begin{split}
A_{2}=& \sum_{i=0}^{k_{n}-1}\int_{0}^{T}\sign(R_{i}-k)R_{i}v(R_{i})\phi(x_{i+1})\bigl(\varphi(x_{i+1})-\varphi(x_{i})\bigr)dt
\\
&+ \sum_{i=k_{n}+1}^{n-1}\int_{0}^{T}\sign(R_{i}-k)R_{i}v(R_{i})\phi(x_{i})\bigl(\varphi(x_{i+1})-\varphi(x_{i})\bigr)dt
\\
&+\int_{0}^{T}\sign(R_{k_{n}}-k)R_{k_{n}}v(R_{k_{n}})\biggl(\int_{x_{k_{n}}}^{x_{k_{n}+1}}\phi\varphi_{x}dx\biggr)dt
\\
&+ \sum_{i=0}^{k_{n}-1}\int_{0}^{T}\sign(R_{i}-k)R_{i}v(R_{i})\biggl(\int_{x_{i}}^{x_{i+1}}\phi'(\tilde{y}_{i,i+1})(x-x_{i+1})\varphi_{x}dx\biggr)dt
\\
&+ \sum_{i=k_{n}+1}^{n-1}\int_{0}^{T}\sign(R_{i}-k)R_{i}v(R_{i})\biggl(\int_{x_{i}}^{x_{i+1}}\phi'(\tilde{x}_{i,i+1})(x-x_{i})\varphi_{x}dx\biggr)dt.
\end{split}
\end{equation}
\item[(P4)] In case (P4), we combine again cases (P1) and (P2), namely we use the Taylor's expansions \eqref{28} and \eqref{29} of $\phi$ respectively for $i\in\{0,\dots,k_{n}-1\}$ and  $i\in\{k_{n}+1,\dots,n-1\}$. Then it follows that
\begin{equation}
\label{145}
\begin{split}
A_{2}=& \sum_{i=0}^{k_{n}-1}\int_{0}^{T}\sign(R_{i}-k)R_{i}v(R_{i})\phi(x_{i})\bigl(\varphi(x_{i+1})-\varphi(x_{i})\bigr)dt
\\
&+ \sum_{i=k_{n}+1}^{n-1}\int_{0}^{T}\sign(R_{i}-k)R_{i}v(R_{i})\phi(x_{i+1})\bigl(\varphi(x_{i+1})-\varphi(x_{i})\bigr)dt
\\
&+\int_{0}^{T}\sign(R_{k_{n}}-k)R_{k_{n}}v(R_{k_{n}})\biggl(\int_{x_{k_{n}}}^{x_{k_{n}+1}}\phi\varphi_{x}dx\biggr)dt
\\
&+ \sum_{i=0}^{k_{n}-1}\int_{0}^{T}\sign(R_{i}-k)R_{i}v(R_{i})\biggl(\int_{x_{i}}^{x_{i+1}}\phi'(\tilde{x}_{i,i+1})(x-x_{i})\varphi_{x}dx\biggr)dt
\\
&+ \sum_{i=k_{n}+1}^{n-1}\int_{0}^{T}\sign(R_{i}-k)R_{i}v(R_{i})\biggl(\int_{x_{i}}^{x_{i+1}}\phi'(\tilde{y}_{i,i+1})(x-x_{i+1})\varphi_{x}dx\biggr)dt.
\end{split}
\end{equation}
\end{itemize}

From now on, let us consider the four cases separately. Putting together \eqref{140}-\eqref{141} and \eqref{142}, we get that in case (P1) it holds
\begin{equation*}
\begin{split}
\int_{0}^T&\int_{\mathbb{R}}\bigl[|\rho^{n}-k|\varphi_{t}+\sign(\rho^{n}-k)\bigl(f(\rho)-f(k)\bigr)\phi\varphi_{x}-\sign(\rho^{n}-k)f(k)\phi'\varphi\bigr]dxdt
\\
=&k\int_{0}^{T}\bigl(v(k)-v(R_{0})\bigr)\phi(x_{0})\varphi(x_{0})dt+ k\int_{0}^{T}\bigl(v_{\max}-v(k)\bigr)\phi(x_{n})\varphi(x_{n})dt
\\
& +\sum_{i=0}^{n-1}\int_{0}^{T}\sign(R_{i}-k)\biggl[-\dot{R}_{i}\biggl(\int_{x_{i}}^{x_{i+1}}\varphi dx\biggr)-k\bigl[\dot{x}_{i}-v(k)\phi(x_{i})\bigr]\varphi(x_{i})
\\
&\hspace{3.9cm}-\bigl[R_{i}(\dot{x}_{i+1}-\dot{x}_{i})-k\bigl(\dot{x}_{i+1}-v(k)\phi(x_{i+1})\bigr)\bigr]\varphi(x_{i+1})\biggr]dt
\\
&+ \sum_{i=0}^{n-1}\int_{0}^{T}\sign(R_{i}-k)R_{i}v(R_{i})\biggl(\int_{x_{i}}^{x_{i+1}}\phi'(\tilde{x}_{i,i+1})(x-x_{i})\varphi_{x}dx\biggr)dt,
\end{split}
\end{equation*}
where
\begin{equation*}
\begin{split}
-R_{i}\bigl(\dot{x}_{i+1}-\dot{x}_{i}\bigr)\varphi(x_{i+1})&=-R_{i}\frac{\dot{x}_{i+1}-\dot{x}_{i}}{x_{i+1}-x_{i}}\biggl(\int_{x_{i}}^{x_{i+1}}\varphi(x_{i+1})dx\biggr)=\dot{R}_{i}\biggl(\int_{x_{i}}^{x_{i+1}}\varphi(x_{i+1})dx\biggr).
\end{split}
\end{equation*}
Moreover, using the definition of $\dot{R}_{i}(t)$, the previous identity becomes
\begin{equation*}
\begin{split}
\int_{0}^T\int_{\mathbb{R}}\bigl[|\rho^{n}-k|\varphi_{t}+\sign(\rho^{n}-k)\bigl(f(\rho)-f(k)\bigr)\phi\varphi_{x}-\sign(\rho^{n}-k)f(k)\phi'\varphi\bigr]dxdt =&kD+E_{1}+E_{2},
\end{split}
\end{equation*}
where
\begin{equation*}
\begin{split}
D:=& \int_{0}^{T}\bigl(v(k)-v(R_{0})\bigr)\phi(x_{0})\varphi(x_{0})dt+ \int_{0}^{T}\bigl(v_{\max}-v(k)\bigr)\phi(x_{n})\varphi(x_{n})dt
\\
& +\sum_{i=0}^{n-1}\int_{0}^{T}\sign(R_{i}-k)\biggl[\bigl[\dot{x}_{i+1}-v(k)\phi(x_{i+1})\bigr]\varphi(x_{i+1})-\bigl[\dot{x}_{i}-v(k)\phi(x_{i})\bigr]\varphi(x_{i})\biggr]dt,
\end{split}
\end{equation*}
\begin{equation*}
E_{1}:= \sum_{i=0}^{n-1}\int_{0}^{T}\sign(R_{i}-k)R_{i}v(R_{i})\biggl(\int_{x_{i}}^{x_{i+1}}\phi'(\tilde{x}_{i,i+1})(x-x_{i})\varphi_{x}dx\biggr)dt
\end{equation*}
and
\begin{equation*}
\begin{split}
E_{2}:=& \sum_{i=0}^{n-1}\int_{0}^{T}\sign(R_{i}-k)\frac{R_{i}^{2}}{\ell_{n}}(\dot{x}_{i+1}-\dot{x}_{i})\biggl(\int_{x_{i}}^{x_{i+1}}\bigl(\varphi(x)-\varphi(x_{i+1})\bigr) dx\biggr)dt.
\end{split}
\end{equation*}
Regarding $E_{1}$, since $\varphi(x,\cdot)$ is a Lipschitz function for every $x\in\mathbb{R}$ and due to Lemma \ref{lem7}, we have that
\begin{equation}
\label{27}
\begin{split}
E_{1}\geq -\ell_{n}L'\sum_{i=0}^{n-1}\int_{0}^{T}\bigl|\varphi(x_{i+1})-\varphi(x_{i})\bigr|dt&\geq -\ell_{n}L'\Lip_{\varphi}\int_{0}^{T}(x_{n}-x_{0})dt\geq - \ell_{n}L'\Lip_{\varphi}T\biggl[\overline{x}_{\max}-\overline{x}_{\min}+2LT\biggr],
\end{split}
\end{equation}
while, using again the Lipschitz condition of $\varphi$ and due to Lemma \ref{lem7} and \eqref{11}, we get that $E_{2}$ satisfies
\begin{equation}
\label{227}
\begin{split}
E_{2}\geq& -\sum_{i=0}^{n-2}\int_{0}^{T}\frac{{R}_{i}^{2}}{\ell_{n}}\bigl|v(R_{i})\phi({x_{i}})-v(R_{i+1})\phi({x_{i+1}})\bigr|\biggl(\int_{x_{i}}^{x_{i+1}}\bigl|\varphi(x)-\varphi(x_{i+1})\bigr|dx\biggr)dt
\\
&-\int_{0}^{T}\frac{{R}_{n-1}^{2}}{\ell_{n}}\bigl|v_{\max}\phi(x_{n})-v(R_{n-1})\phi({x_{n-1}})\bigr|\biggl(\int_{x_{n-1}}^{x_{n}}\bigl|\varphi(x)-\varphi(x_{n})\bigr|dx\biggr)dt
\\
\geq& -\ell_{n}\Lip_{\varphi}\biggl[\sum_{i=0}^{n-2}\int_{0}^{T}\bigl|v(R_{i})\phi({x_{i}})-v(R_{i+1})\phi({x_{i+1}})\bigr|dt+\int_{0}^{T}\bigl|v_{\max}\phi({x_{n}})-v(R_{n-1})\phi({x_{n-1}})\bigr|dt\biggr]
\\
\geq & -\ell_{n}\Lip_{\varphi}\biggl[\sum_{i=0}^{n-2}\int_{0}^{T}v(R_{i})\bigl|\phi({x_{i}})-\phi({x_{i+1}})\bigr|dt+\sum_{i=0}^{n-2}\int_{0}^{T}\phi({x_{i+1}})\bigl|v(R_{i})-v(R_{i+1})\bigr|dt+LT\biggr]
\\
\geq & -\ell_{n}\Lip_{\varphi}\biggl[L'\int_{0}^{T}(x_{n}-{x_{0}})dt+\Lip_{v}||\phi||_{L^{\infty}(\mathbb{R})}\int_{0}^{T}\sum_{i=0}^{n-2}|R_{i}-R_{i+1}|dt+LT\biggr]
\\
\geq & -\ell_{n}\Lip_{\varphi}\biggl[L'T\bigl[\overline{x}_{\max}-\overline{x}_{\min}+2LT\bigr]+\Lip_{v}||\phi||_{L^{\infty}(\mathbb{R})}\bigl(\TV[\overline{\rho}]+\alpha t+\beta e^{L'T}\bigr) e^{[\gamma T(1+T)+\zeta e^{L'T}]}+LT\biggr].
\end{split}
\end{equation}
Putting together the previous two estimates, we hence get that $E_{1}+E_{2}\geq -c\, \ell_{n}$ for some positive constant $c$ independent on $n$ and so, since the right-hand side of this inequality tends to zero as $n\rightarrow +\infty$, to conclude the proof it is sufficient to show that $D$ is non-negative. From a direct calculation we remark that
\begin{equation}
\label{45}
\begin{split}
\sum_{i=0}^{n-1}&\sign(R_{i}-k)\biggl[\bigl[\dot{x}_{i+1}-v(k)\phi(x_{i+1})\bigr]\varphi(x_{i+1})-\bigl[\dot{x}_{i}-v(k)\phi(x_{i})\bigr]\varphi(x_{i})\biggr]
\\
=&\sum_{i=1}^{n-1}\bigl[\sign(R_{i-1}-k)-\sign(R_{i}-k)\bigr]\bigl(v(R_{i})-v(k)\bigr)\phi(x_{i})\varphi(x_{i})
\\
&-\sign(R_{0}-k)\bigl(v(R_{0})-v(k)\bigr)\phi(x_{0})\varphi(x_{0})+\sign(R_{n-1}-k)\bigl(v_{\max}-v(k)\bigr)\phi(x_{n})\varphi(x_{n}),
\end{split}
\end{equation}
therefore $D$ can be rewritten as
\begin{equation*}
\begin{split}
D= \int_{0}^{T}D_{0}(t)dt+\sum_{i=1}^{n-1}\int_{0}^{T}D_{i}(t)dt+\int_{0}^{T}D_{n}(t)dt,
\end{split}
\end{equation*}
where
\begin{equation*}
D_{0}:= \bigl(1+\sign(R_{0}-k)\bigr)\bigl(v(k)-v(R_{0})\bigr)\phi(x_{0})\varphi(x_{0}), \quad D_{n}:= \bigr(1+\sign(R_{n-1}-k)\bigr)\bigl(v_{\max}-v(k)\bigr)\phi(x_{n})\varphi(x_{n})
\end{equation*}
and
\begin{equation*}
D_{i}:=\bigl[\sign(R_{i-1}-k)-\sign(R_{i}-k)\bigr]\bigl(v(R_{i})-v(k)\bigr)\phi(x_{i})\varphi(x_{i}) \quad \text{for} \ i\in \{1,\dots,n-1\}.
\end{equation*}
Concerning $D_{0}$ and $D_{n}$, we have two different sub-cases, namely
\begin{equation*}
D_{0}\begin{cases} =0\quad &\text{if}\ R_{0}\leq k,
\\ \geq 0 &\text{if} \ R_{0}<k,
\end{cases} \quad \text{and} \quad
D_{n}\begin{cases} =0\quad &\text{if}\ R_{n-1}< k,
\\ \geq 0 \quad &\text{if} \ R_{n-1}\geq k.
\end{cases}
\end{equation*}
Turning to $D_{i}$ with $i\in\{1,\dots,n-1\}$, since $v$ is non-increasing and $\phi$, $\varphi$ are non-negative, after a simple calculation we get that
\begin{equation*}
D_{i}\begin{cases} =0\quad &\text{if}\ R_{i-1}>k\ \text{and}\ R_{i}>k \ \text{or if}\ R_{i-1}<k\ \text{and}\ R_{i}<k \ \text{or if} \ R_{i}=k,
\\ \geq 0 \quad &\text{otherwise}
\end{cases}
\end{equation*}
and this concludes the proof of \eqref{48} in case (P1).

For case (P2), using instead the first order Taylor's expansion \eqref{29}, we can proceed in a symmetric way to get the validity of \eqref{48} also in this case (the details are left to the reader).

Let us now consider the case (P3). Putting together \eqref{140}-\eqref{141} and \eqref{144}, we get that
\begin{equation*}
\begin{split}
\int_{0}^T&\int_{\mathbb{R}}\bigl[|\rho^{n}-k|\varphi_{t}+\sign(\rho^{n}-k)\bigl(f(\rho)-f(k)\bigr)\phi\varphi_{x}-\sign(\rho^{n}-k)f(k)\phi'\varphi\bigr]dxdt= kD+E_{1}^{3}+E_{2}^{3}+F,
\end{split}
\end{equation*}
where
\begin{equation*}
\begin{split}
D:=&\int_{0}^{T}\bigl(v(k)-v_{\max}\bigr)\phi(x_{0})\varphi(x_{0})dt+ \int_{0}^{T}\bigl(v_{\max}-v(k)\bigr)\phi(x_{n})\varphi(x_{n})dt
\\
&+\sum_{i=0}^{n-1}\int_{0}^{T}\sign(R_{i}-k)\biggl[\bigl[\dot{x}_{i+1}-v(k)\phi(x_{i+1})\bigr ]\varphi(x_{i+1})-\bigl[\dot{x}_{i}-v(k)\phi(x_{i})\bigr]\varphi(x_{i})\biggr]dt,
\end{split}
\end{equation*}
\begin{equation*}
\begin{split}
E_{1}^{3}:=& \sum_{i=0}^{k_{n}-1}\int_{0}^{T}\sign(R_{i}-k)R_{i}v(R_{i})\biggl(\int_{x_{i}}^{x_{i+1}}\phi'(\tilde{y}_{i,i+1})(x-x_{i+1})\varphi_{x}dx\biggr)dt
\\
&+ \sum_{i=k_{n}+1}^{n-1}\int_{0}^{T}\sign(R_{i}-k)R_{i}v(R_{i})\biggl(\int_{x_{i}}^{x_{i+1}}\phi'(\tilde{x}_{i,i+1})(x-x_{i})\varphi_{x}dx\biggr)dt,
\end{split}
\end{equation*}
\begin{equation*}
\begin{split}
E_{2}^{3}:= &-\!\sum_{i=0}^{k_{n}-1}\!\int_{0}^{T}\!\!\sign(R_{i}-k)\dot{R}_{i}\biggl(\int_{x_{i}}^{x_{i+1}}\!\!\!\!\bigl(\varphi-\varphi(x_{i})\bigr) dx\biggr)dt-\!\!\!\sum_{i=k_{n}+1}^{n-1}\!\int_{0}^{T}\!\!\!\!\sign(R_{i}-k)\dot{R}_{i}\biggl(\int_{x_{i}}^{x_{i+1}}\!\!\!\!\!\bigl(\varphi-\varphi(x_{i+1})\bigr) dx\biggr)dt
\end{split}
\end{equation*}
and
\begin{equation*}
\begin{split}
F:= &-\int_{0}^{T}\sign(R_{k_{n}}-k)\dot{R}_{k_{n}}\biggl(\int_{x_{k_{n}}}^{x_{k_{n}+1}}\varphi dx\biggr) dt-\int_{0}^{T}\sign(R_{k_{n}}-k)R_{k_{n}}\bigl[\dot{x}_{k_{n}+1}\varphi(x_{k_{n}+1})-\dot{x}_{k_{n}}\varphi(x_{k_{n}})\bigr] dt
\\
&+\int_{0}^{T}\sign(R_{k_{n}}-k)R_{k_{n}}v(R_{k_{n}})\biggl(\int_{x_{k_{n}}}^{x_{k_{n}+1}}\phi(x)\varphi_{x}(x)dx\biggr)dt.
\end{split}
\end{equation*}
Rearranging the indices properly, we can prove that $E_{1}^{3}$ and $E_{2}^{3}$ satisfy a similar estimate to \eqref{27} and \eqref{227} respectively. Concerning $F$, we first remark that it holds
\begin{equation*}
\label{43}
\begin{split}
-R_{k_{n}}\bigl[\dot{x}_{k_{n}+1}\varphi(x_{k_{n}+1})-\dot{x}_{k_{n}}\varphi(x_{k_{n}})\bigr]=&-R_{k_{n}}\dot{x}_{k_{n}+1}\bigl(\varphi(x_{k_{n}+1})-\varphi(x_{k_{n}})\bigr)-R_{k_{n}}\varphi(x_{k_{n}})(\dot{x}_{k_{n}+1}-\dot{x}_{k_{n}})
\\
=&-R_{k_{n}}\dot{x}_{k_{n}+1}\bigl(\varphi(x_{k_{n}+1})-\varphi(x_{k_{n}})\bigr)+\dot{R}_{k_{n}}\biggl(\int_{x_{k_{n}}}^{x_{k_{n+1}}}\varphi(x_{k_{n}})dx\biggr).
\end{split}
\end{equation*}
Moreover, since the first order Taylor's expansion of $\phi$ at $0$ in the intervals $(x_{k_{n}},0)$ and $(0,x_{k_{n}+1})$ implies
\begin{equation*}
\label{37}
\begin{split}
\int_{x_{k_{n}}}^{x_{k_{n}+1}}\phi(x)\varphi_{x}(x)dx=&\phi(0)\bigl(\varphi(x_{k_{n}+1})-\varphi(x_{k_{n}})\bigr)+\int_{x_{k_{n}}}^{0}\phi'(\tilde{y}_{k_{n}})x\varphi_x(x)dx+\int_{0}^{x_{k_{n}+1}}\phi'(\tilde{x}_{k_{n}+1})x\varphi_x(x)dx
\end{split}
\end{equation*}
for some $\tilde{y}_{k_{n}}\in (x_{k_{n}},0)$ and $\tilde{x}_{k_{n}+1}\in (0,x_{k_{n}+1})$, then, using the Lipschitz condition on $\varphi$ and \eqref{130}, we get
\begin{equation*}
\begin{split}
F\geq &-\int_{0}^{T}|\dot{R}_{k_{n}}|\biggl(\int_{x_{k_{n}}}^{x_{k_{n}+1}}\bigl|\varphi-\varphi(x_{k_{n}})\bigr| dx\biggr)dt-\int_{0}^{T}R_{k_{n}}\dot{x}_{k_{n}+1}\bigl|\varphi(x_{k_{n}+1})-\varphi(x_{k_{n}})\bigr|dt
\\
&-\int_{0}^{T}R_{k_{n}}v(R_{k_{n}})|\phi(0)|\bigl|\varphi(x_{k_{n}+1})-\varphi(x_{k_{n}})\bigr|dt-\int_{0}^{T}R_{k_{n}}v(R_{k_{n}})\biggl|\int_{x_{k_{n}}}^{0}\phi'(\tilde{y}_{k_{n}})x\varphi_xdx\biggr|dt
\\
&-\int_{0}^{T}R_{k_{n}}v(R_{k_{n}})\biggl|\int_{0}^{x_{k_{n}+1}}\phi'(\tilde{x}_{k_{n}+1})x\varphi_xdx\biggr|dt
\\
\geq & - \ell_{n}\Lip_{\varphi}T\bigl[4L+L' \bigl(\overline{x}_{\max}-\overline{x}_{\min}+2LT\bigr)\bigr].
\end{split}
\end{equation*}
As a consequence it follows that $E_{1}^{3}+E_{2}^{3}+F\geq -c\, \ell_{n}$ for some constant $c$ independent on $n$ and hence it remains to prove as before that $D$ is non-negative. From \eqref{45} we get that
\begin{equation*}
\begin{split}
D=&\int_{0}^{T}\bigl(1+\sign(R_{0}-k)\bigr)\bigl(v(k)-v_{\max}\bigr)\phi(x_{0})\varphi(x_{0})dt
\\
&+\sum_{i=1}^{k_{n}}\int_{0}^{T}\bigl[\sign(R_{i-1}-k)-\sign(R_{i}-k)\bigr]\bigl(v(R_{i-1})-v(k)\bigr)\phi(x_{i})\varphi(x_{i})dt
\\
&+\sum_{i=k_{n}+1}^{n-1}\int_{0}^{T}\bigl[\sign(R_{i-1}-k)-\sign(R_{i}-k)\bigr]\bigl(v(R_{i})-v(k)\bigr)\phi(x_{i})\varphi(x_{i})dt
\\
&+\int_{0}^{T}\bigl(1+\sign(R_{n-1}-k)\bigr)\bigl(v_{\max}-v(k)\bigr)\phi(x_{n})\varphi(x_{n})dt
\end{split}
\end{equation*}
and this implies that $D\geq 0$, since we can estimate each term as we did in the previous cases.

Turning to the last case (P4), we combine again cases (P1) and (P2), namely we use the Taylor's expansions \eqref{28} and \eqref{29} of $\phi$ respectively for $i\in\{0,\dots,k_{n}-1\}$ and  $i\in\{k_{n}+1,\dots,n-1\}$. This implies, putting together \eqref{140}-\eqref{141} and \eqref{145}, that
\begin{equation*}
\begin{split}
\int_{0}^T&\int_{\mathbb{R}}\bigl[|\rho^{n}-k|\varphi_{t}+\sign(\rho^{n}-k)\bigl(f(\rho)-f(k)\bigr)\phi\varphi_{x}-\sign(\rho^{n}-k)f(k)\phi'\varphi\bigr]dxdt
=kD+E_{1}^{4}+E_{2}^{4}+F,
\end{split}
\end{equation*}
where
\begin{equation*}
\begin{split}
D:=&\int_{0}^{T}\bigl(v(k)-v(R_{0})\bigr)\phi(x_{0})\varphi(x_{0})dt+ \int_{0}^{T}\bigl(v(R_{n-1})-v(k)\bigr)\phi(x_{n})\varphi(x_{n})dt
\\
&+\sum_{i=0}^{n-1}\int_{0}^{T}\sign(R_{i}-k)\biggl[\bigl[\dot{x}_{i+1}-v(k)\phi(x_{i+1})\bigr]\varphi(x_{i+1})-\bigl[\dot{x}_{i}-v(k)\phi(x_{i})\bigr]\varphi(x_{i})\biggr]dt,
\end{split}
\end{equation*}
\begin{equation*}
\begin{split}
E_{1}^{4}:=& \sum_{i=0}^{k_{n}-1}\int_{0}^{T}\sign(R_{i}-k)R_{i}v(R_{i})\biggl(\int_{x_{i}}^{x_{i+1}}\phi'(\tilde{x}_{i,i+1})(x-x_{i})\varphi_{x}dx\biggr)dt
\\
&+ \sum_{i=k_{n}+1}^{n-1}\int_{0}^{T}\sign(R_{i}-k)R_{i}v(R_{i})\biggl(\int_{x_{i}}^{x_{i+1}}\phi'(\tilde{y}_{i,i+1})(x-x_{i+1})\varphi_{x}dx\biggr)dt,
\end{split}
\end{equation*}
\begin{equation*}
\begin{split}
E_{2}^{4}:= &\!-\!\sum_{i=0}^{k_{n}-1}\!\int_{0}^{T}\!\!\!\sign(R_{i}-k)\dot{R}_{i}\biggl(\int_{x_{i}}^{x_{i+1}}\!\!\!\!\bigl(\varphi-\varphi(x_{i+1})\bigr) dx\biggr)dt-\!\!\!\!\sum_{i=k_{n}+1}^{n-1}\int_{0}^{T}\!\!\!\sign(R_{i}-k)\dot{R}_{i}\biggl(\int_{x_{i}}^{x_{i+1}}\!\!\!\!\bigl(\varphi-\varphi(x_{i})\bigr) dx\biggr)dt,
\end{split}
\end{equation*}
while $F$ is the same term defined in case (P3). Arguing as in the previous cases, we can prove that $E_{1}^{4}$ and $E_{2}^{4}$ satisfy a similar estimate to \eqref{27} and \eqref{227} respectively, so it follows that $E_{1}^{4}+E_{2}^{4}+F\geq -c\, \ell_{n}$ for some constant $c$ independent on $n$ and hence to conclude the proof we need to show as before that $D$ is non-negative. Using again \eqref{45} we get that
\begin{equation*}
\begin{split}
D=\ &\int_{0}^{T}\bigl(1+\sign(R_{0}-k)\bigr)\bigl(v(k)-v(R_{0})\bigr)\phi(x_{0})\varphi(x_{0})dt
\\
&+\sum_{i=1}^{k_{n}}\int_{0}^{T}\bigl[\sign(R_{i-1}-k)-\sign(R_{i}-k)\bigr]\bigl(v(R_{i})-v(k)\bigr)\phi(x_{i})\varphi(x_{i})dt
\\
&+\sum_{i=k_{n}+1}^{n-1}\int_{0}^{T}\bigl[\sign(R_{i-1}-k)-\sign(R_{i}-k)\bigr]\bigl(v(R_{i-1})-v(k)\bigr)\phi(x_{i})\varphi(x_{i})dt
\\
&+\int_{0}^{T}\bigl(1+\sign(R_{n-1}-k)\bigr)\bigl(v(R_{n-1})-v(k)\bigr)\phi(x_{n})\varphi(x_{n})dt
\end{split}
\end{equation*}
and hence it follows that $D\geq 0$ also in this last case, since we can estimate each term as we did before.

Now it remains to prove that $\rho$ satisfies the entropy condition, that is
\begin{equation*}
\begin{split}
\int_{0}^{T}\int_{\mathbb{R}}\biggl[&|\rho(x,t)-k|\varphi_{t}(x,t)+\sign(\rho(x,t)-k)\bigl(f(\rho(x,t))-f(k)\bigr)\phi(x)\varphi_{x}(x,t)
\\
&\hspace{2.9cm}-\sign(\rho(x,t)-k)f(k)\phi'(x)\varphi(x,t)\biggr]dxdt\geq 0
\end{split}
\end{equation*}
for every $k\geq0$ and every non-negative $\varphi\in C_{c}^{\infty}(\mathbb{R}\times (0,T))$.
\\
We first notice that the previous inequality is a direct consequence of \eqref{48}: we need only to show that it is possible to interchange the limit and the integrals. The convergence of $\{\rho^{n}\}_{n\in\mathbb{N}}$ to $\rho$ almost everywhere and in $L^{1}$ on $\mathbb{R}\times [0,T]$ implies that
\begin{equation}
\label{60}
\lim_{n\rightarrow +\infty}\int_{0}^{T}\int_{\mathbb{R}}|\rho^{n}-k|\varphi_{t}dxdt= \int_{0}^{T}\int_{\mathbb{R}}|\rho-k|\varphi_{t}dxdt
\end{equation}
and moreover, since $g(\mu):= \sign (\mu-k)\bigl(f(\mu)-f(k)\bigr)$ is a continuous function, we also have that
\begin{equation}
\label{61}
\begin{split}
\lim_{n\rightarrow +\infty}&\int_{0}^{T}\int_{\mathbb{R}}\sign (\rho^{n}-k)\bigl(f(\rho^{n})-f(k)\bigr)\phi\varphi_{x}dxdt= \int_{0}^{T}\int_{\mathbb{R}}\sign (\rho-k)\bigl(f(\rho)-f(k)\bigr)\phi\varphi_{x}dxdt.
\end{split}
\end{equation}
Therefore it remains to show that
\begin{equation}
\label{46}
\lim_{n\rightarrow +\infty}\int_{0}^{T}\int_{\mathbb{R}}\sign(\rho^{n}-k)f(k)\phi'\varphi dxdt=\int_{0}^{T}\int_{\mathbb{R}}\sign(\rho-k)f(k)\phi'\varphi dxdt,
\end{equation}
where, since $g(\mu)=\sign(\mu-k)$ is a discontinuous function, we can't interchange the limit and the integrals directly. To overcome this problem, we need to introduce $\eta_{\varepsilon}^{\pm}$, two Lipschitz approximations of the $\sign$ function  such that
\begin{equation}
\label{49}
\sign(z)-\eta_{\varepsilon}^{+}(z)\geq 0 \quad \text{and} \quad \sign(z)-\eta_{\varepsilon}^{-}(z)\leq 0 \quad \text{for all}\ z\in\mathbb{R}.
\end{equation}
Let us denote $M:=||\phi'||_{L^{\infty}(\mathbb{R})}$ from now on. We first remark that \eqref{49} implies
\begin{equation*}
\begin{split}
\int_{0}^{T}\int_{\mathbb{R}}\!\sign(\rho^{n}-k)f(k)\phi'\varphi dxdt
=&\int_{0}^{T}\int_{\mathbb{R}}\sign(\rho^{n}-k)f(k)(\phi'-M)\varphi dxdt+\int_{0}^{T}\int_{\mathbb{R}}\!\sign(\rho^{n}-k) f(k)M\varphi dxdt
\\
\leq&\int_{0}^{T}\int_{\mathbb{R}}\eta_{\varepsilon}^{+}(\rho^{n}-k)f(k)(\phi'-M)\varphi dxdt+ \int_{0}^{T}\int_{\mathbb{R}}\eta_{\varepsilon}^{-}(\rho^{n}-k)f(k)M\varphi dxdt.
\end{split}
\end{equation*}
On the other hand, from the Lipschitz condition of $\eta_{\varepsilon}^{\pm}$ and using again the convergence of $\{\rho^{n}\}_{n\in\mathbb{N}}$ to $\rho$, it follows that
\begin{equation}
\label{50}
\begin{split}
\lim_{n\rightarrow +\infty}& \int_{0}^{T}\int_{\mathbb{R}}\bigl[\eta_{\varepsilon}^{+}(\rho^{n}-k)-\eta_{\varepsilon}^{+}(\rho-k)\bigr]f(k)(\phi'-M)\varphi dxdt
\\
&\leq f(k)||\varphi||_{L
^{\infty}(\mathbb{R}\times[0,T])}\lim_{n\rightarrow +\infty}\int_{0}^{T}\int_{\mathbb{R}}|\eta_{\varepsilon}^{+}(\rho^{n}-k)-\eta_{\varepsilon}^{+}(\rho-k)||\phi'-M| dxdt
\\
&\leq 2f(k)M||\varphi||_{L
^{\infty}(\mathbb{R}\times[0,T])}\Lip_{\eta_{\varepsilon}^{+}}\lim_{n\rightarrow +\infty}\int_{0}^{T}\int_{\mathbb{R}}|\rho^{n}-\rho|dxdt= 0
\end{split}
\end{equation}
and analogously that
\begin{equation}
\label{51}
\begin{split}
\lim_{n\rightarrow +\infty}& \int_{0}^{T}\int_{\mathbb{R}}\bigl[\eta_{\varepsilon}^{-}(\rho^{n}-k)-\eta_{\varepsilon}^{-}(\rho-k)\bigr]f(k)M\varphi dxdt
\\
&\leq f(k)M||\varphi||_{L
^{\infty}(\mathbb{R}\times[0,T])}\lim_{n\rightarrow +\infty}\int_{0}^{T}\int_{\mathbb{R}}|\eta_{\varepsilon}^{-}(\rho^{n}-k)-\eta_{\varepsilon}^{-}(\rho-k)|dxdt
\\
&\leq f(k) M||\varphi||_{L
^{\infty}(\mathbb{R}\times[0,T])}\Lip_{\eta_{\varepsilon}^{-}}\lim_{n\rightarrow +\infty}\int_{0}^{T}\int_{\mathbb{R}}|\rho^{n}-\rho|dxdt= 0.
\end{split}
\end{equation}
Combining the previous three estimates, we hence get
\begin{equation*}
\begin{split}
\limsup_{n\rightarrow +\infty}&\int_{0}^{T}\int_{\mathbb{R}}\sign(\rho^{n}-k)f(k)\phi'\varphi dxdt\leq  \int_{0}^{T}\int_{\mathbb{R}}\bigl[\eta_{\varepsilon}^{+}(\rho-k)(\phi'-M)+\eta_{\varepsilon}^{-}(\rho-k)M\bigr]f(k)\varphi dxdt
\end{split}
\end{equation*}
and, since it holds
\begin{equation*}
\bigl[\eta_{\varepsilon}^{+}(\rho-k)(\phi'-M)+\eta_{\varepsilon}^{-}(\rho-k)M\bigr]f(k)\varphi
\leq 3f(k)M\varphi\in L^{1}(\mathbb{R}\times [0,T]),
\end{equation*}
then we can apply the dominated convergence theorem and pass to the limit in $\varepsilon$, which implies
\begin{equation}
\label{52}
\limsup_{n\rightarrow +\infty}\int_{0}^{T}\int_{\mathbb{R}}\sign(\rho^{n}-k)f(k)\phi'\varphi dxdt\leq \int_{0}^{T}\int_{\mathbb{R}}\sign(\rho-k)f(k)\phi'\varphi dxdt.
\end{equation}
Proceeding in a symmetric way we notice that
\begin{equation*}
\begin{split}
\int_{0}^{T}\int_{\mathbb{R}}\sign(\rho^{n}-k)f(k)\phi'\varphi dxdt \geq& \int_{0}^{T}\int_{\mathbb{R}}\eta_{\varepsilon}^{-}(\rho^{n}-k)f(k)(\phi'-M)\varphi dxdt
\\
&+ \int_{0}^{T}\int_{\mathbb{R}}\eta_{\varepsilon}^{+}(\rho^{n}-k)f(k)M\varphi dxdt
\end{split}
\end{equation*}
and, using again \eqref{50}, \eqref{51} and since
\begin{equation*}
\bigl[\eta_{\varepsilon}^{-}(\rho-k)(\phi'-M)+\eta_{\varepsilon}^{+}(\rho-k)M\bigr]f(k)\varphi
\leq 3f(k)M\varphi\in L^{1}(\mathbb{R}\times [0,T]),
\end{equation*}
therefore we can apply as before the dominated convergence theorem in $\varepsilon$ and we get
\begin{equation*}
\liminf_{n\rightarrow +\infty}\int_{0}^{T}\int_{\mathbb{R}}\sign(\rho^{n}-k)f(k)\phi'\varphi dxdt\geq  \int_{0}^{T}\int_{\mathbb{R}}\sign(\rho-k)f(k)\phi'\varphi dxdt.
\end{equation*}
Combining the last inequality with \eqref{52}, we hence get \eqref{46} and this implies, together with \eqref{48}, \eqref{60} and \eqref{61}, that $\rho$ is a weak solution to \eqref{sis2} satisfying the entropy condition. 

Regarding the strong $L^1$ continuity w.r.t $t$ near $t=0$, we observe that, for an arbitrary $\sigma>0$,
\[\|\rho(\cdot,t)-\overline{\rho}\|_{L^1}\leq \|\rho^n(\cdot,t)-\rho^n(\cdot,0)\|_{L^1} + 2\sigma\]
for some $n$ large enough. This is due to the uniform $BV$ estimates on $\rho^n(t)$ and $\rho^n(0)$ which imply strong $L^1$ compactness for both $\rho^n(t)$ and $\rho^n(0)$. Now, the following interpolation inequality can be recovered as a special case of \cite[Theorem 6.4.1]{hormander_book}, by approximation of $BV$ functions via smooth functions in $L^1$ (see e.g. \cite[Proposition 3.7]{AFP}) and by the representation of the $1$-Wasserstein distance in terms of primitive variables, using the fact that the support of $\rho^n$ is uniformly bounded w.r.t. $n$ on finite time intervals, we omit the details:
\[\|\rho^n(\cdot,t)-\rho^n(\cdot,0)\|_{L^1}\leq C \left(TV[\rho^n(\cdot,t)] + TV[\rho^n(\cdot,0)]\right)^{1/2} W_1(\rho^n(\cdot,t),\rho^n(\cdot,0))^{1/2}.\]
Hence, the result in Proposition \ref{prop3} implies the existence of a constant $\bar{C}>0$ independent of $n$ such that
\[\|\rho(\cdot,t)-\overline{\rho}\|_{L^1}\leq C t^{1/2} + \sigma,\]
and the arbitrariness of $\sigma$ shows that $\rho(\cdot,t)\rightarrow \overline{\rho}$ strongly in $L^1$ as $t\searrow 0$.

Finally, to conclude that $\rho$ is the unique entropy solution we notice that $\overline{\rho}\in L^{1}(\mathbb{R})\cap L^{\infty}(\mathbb{R})\cap BV(\mathbb{R})$ due to assumption (I), $f:= \rho v(\rho)\in Lip_{loc}(\mathbb{R}_{+})$ since $v\in C^{1}(\mathbb{R}_{+})$ by assumption (V), $\psi:=\phi$ is  in $W_{loc}^{1,1}(\mathbb{R})\cap C(\mathbb{R})$ and satisfies $\psi, \psi' \in L^{\infty}(\mathbb{R})$ due to assumption (P). As a consequence, we can apply theorem \ref{thmKR} and this concludes the proof of our main result.
\end{proof}

\section*{Acknowledgements}

The authors acknowledge financial support from the EU-funded Erasmus Mundus programme ‘MathMods - Mathematical models in engineering: theory, methods, and applications’ at the University of L’Aquila, from the local fund of the University of L'Aquila
‘DP-LAND’ (Deterministic Particles for Local And Nonlocal Dynamics) and from the Italian FFABR (Fondo di Finanziamento per le Attivit\`{a} di Base della Ricerca, Italian MIUR) funding scheme. We thank Emanuela Radici for useful suggestions on the readability of the manuscript and for proofreading it.

\bibliography{difsti_new1}

\begin{thebibliography}{10}

\bibitem{AFP}
L~Ambrosio, N.~Fusco, and D.~Pallara.
\newblock {\em Functions of bounded variation and free discontinuity problems}.
\newblock Mathematical Monographs. Oxford University Press, 2000.

\bibitem{Aw}
A.~Aw, A.~Klar, T.~Materne, and M.~Rascle.
\newblock Derivation of continuum traffic flow models from microscopic
  follow-the-leader models.
\newblock {\em SIAM Journal on Applied Mathematics}, 63(1):\ 259--278, 2002.

\bibitem{ARZ}
A.~Aw and M.~Rascle.
\newblock Resurrection of 'second order' models of traffic flow.
\newblock {\em SIAM Journal on Applied Mathematics}, 60(3):\ 916--938, 2000.

\bibitem{degond}
F.~Berthelin, P.~Degond, M.~Delitala, and M.~Rascle.
\newblock A model for the formation and evolution of traffic jams.
\newblock {\em Archive for Rational Mechanics and Analysis}, 187(2):\ 185--220,
  2008.

\bibitem{bergoa}
F.~Berthelin and P.~Goatin.
\newblock Particle approximation of a constrained model for traffic flow.
\newblock {\em Nonlinear Differential Equations and Applications NoDEA},
  24(5)(55), 2017.

\bibitem{buerger_et_al}
F.~Betancourt, R.~B\"{u}rger, K.~H. Karlsen, and E.~M. Tory.
\newblock On nonlocal conservation laws modelling sedimentation.
\newblock {\em Nonlinearity}, 24(3):\ 855--885, 2011.

\bibitem{choi}
J.~A. Carrillo, Y.-P. Choi, and M.~Hauray.
\newblock The derivation of swarming models: mean-field limit and {W}asserstein
  distances.
\newblock {\em Collective Dynamics from Bacteria to Crowds. CISM International
  Centre for Mechanical Sciences}, 553:\ 1--46, 2014.

\bibitem{CDFT_ANS}
J.~A. Carrillo, M.~Di~Francesco, and G.~Toscani.
\newblock Condensation phenomena in nonlinear drift equations.
\newblock {\em Annali della Scuola Normale Superiore di Pisa, Classe di
  Scienze}, 5(15):\ 145--171, 2016.

\bibitem{dafermos_book}
C.~M. Dafermos.
\newblock {\em Hyperbolic conservation laws in continuum physics}, volume 325
  of {\em Grundlehrender Mathematischen Wissenschaften [Fundamental Principles
  of Mathematical Sciences]}.
\newblock Springer-Verlag, Berlin, 2010.

\bibitem{DFFRad}
M.~Di~Francesco, S.~Fagioli, and E.~Radici.
\newblock Deterministic particle approximation for nonlocal transport equations
  with nonlinear mobility.
\newblock {\em Journal of Differential Equations}, 266(5):\ 2830--2868, 2019.

\bibitem{DFFR_ARZ}
M.~Di~Francesco, S.~Fagioli, and M.~D. Rosini.
\newblock Many particle approximation for the aw-rascle-zhang second order
  model for vehicular traffic.
\newblock {\em Mathematical Biosciences and Engineering}, 14:127--141, 2016.

\bibitem{DFRF}
M.~Di~Francesco, S.~Fagioli, and M.~D. Rosini.
\newblock Deterministic particle approximation of scalar conservation laws.
\newblock {\em Bollettino dell'Unione Matematica Italiana}, 10(3):\ 487--501,
  2017.

\bibitem{DFFRR1}
M.~Di~Francesco, S.~Fagioli, M.~D. Rosini, and G.~Russo.
\newblock Deterministic particle approximation of the {H}ughes model in one
  space dimension.
\newblock {\em Kinetic and Related Models}, 10(1):\ 215--237, 2017.

\bibitem{DFFRR2}
M.~Di~Francesco, S.~Fagioli, M.~D. Rosini, and G.~Russo.
\newblock Follow-the-leader approximations of macroscopic models for vehicular
  and pedestrian flows.
\newblock {\em Active Particles}, 1:\ 333--378, 2017.

\bibitem{DFR}
M.~Di~Francesco and M.~D. Rosini.
\newblock Rigorous derivation of nonlinear scalar conservation laws from
  follow-the-leader type models via many particle limit.
\newblock {\em Archive for Rational Mechanics and Analysis}, 217(3):\ 831--871,
  2015.

\bibitem{FRad}
S.~Fagioli and E.~Radici.
\newblock Solutions to aggregation–diffusion equations with nonlinear
  mobility constructed via a deterministic particle approximation.
\newblock {\em Mathematical Models and Methods in Applied Sciences}, 28(9):\
  1801--1829, 2018.

\bibitem{ferrari1}
P.~A. Ferrari.
\newblock Shock fluctuations in asymmetric simple exclusion.
\newblock {\em Probabilty Theory and Related Fields}, 91(1):\ 81--101, 1992.

\bibitem{ferrari2}
P.~L. Ferrari and P.~Nejjar.
\newblock Shock fluctuations in flat tasep under critical scaling.
\newblock {\em Journal of Statistical Physics}, 160(4):\ 985--1004, 2015.

\bibitem{gosse}
L.~Gosse and G.~Toscani.
\newblock Identification of asymptotic decay to self-similarity for
  one-dimensional filtration equations.
\newblock {\em SIAM Journal on Numerical Analysis}, 43(6):\ 2590--2606, 2006.

\bibitem{vara1}
M.~Z. Guo, G.~C. Papanicolaou, and S.~R.~S. Varadhan.
\newblock Nonlinear diffusion limit for a system with nearest neighbor
  interactions.
\newblock {\em Communications in Mathematical Physics}, 118(1):\ 31--59, 1988.

\bibitem{HolRisebro2}
H.~Holden and N.~H. Risebro.
\newblock The continuum limit of follow-the-leader models. a short proof.
\newblock {\em Discrete \& Continuous Dynamical Systems-A}, 38(2):\ 715--722,
  2018.

\bibitem{HolRisebro}
H.~Holden and N.~H. Risebro.
\newblock Follow-the-leader models can be viewed as a numerical approximation
  to the lighthill-whitham-richards model for traffic flow.
\newblock {\em Networks \& Heterogeneous Media}, 13(3):\ 409--421, 2018.

\bibitem{hormander_book}
L.~H\"{o}rmander.
\newblock {\em Lectures on Nonlinear Hyperbolic Differential Equations},
  volume~26 of {\em Mathématiques et Applications}.
\newblock Springer, 1997.

\bibitem{glaciers}
K.~H. Karlsen and K.-A. Lie.
\newblock An unconditionally stable splitting scheme for a class of nonlinear
  parabolic equations.
\newblock {\em IMA Journal of Numerical Analysis}, 19(4):\ 609--635.

\bibitem{KarlsenRisebro}
K.~H. Karlsen and N.~H. Risebro.
\newblock On the uniqueness and stability of entropy solutions of nonlinear
  degenerate parabolic equations with rough coefficients.
\newblock {\em Discrete and Continuous Dynamical Systems}, 9(5):\ 1081--1104,
  2003.

\bibitem{kruzkov}
S.N. Kru{\v{z}}kov.
\newblock First order quasilinear equations in several independent variables.
\newblock {\em Mathematics of the USSR-Sbornik}, 10(2):\ 217--243, 1970.

\bibitem{liggett}
T.~M. Liggett.
\newblock {\em Interacting particle systems}, volume 276 of {\em Grundlehren
  der Mathematischen Wissenschaften [Fundamental Principles of Mathematical
  Sciences]}.
\newblock Springer-Verlag, New York, 1985.

\bibitem{LW}
M.J. Lighthill and G.B. Whitham.
\newblock On kinematic waves. {I}{I}. {A} theory of traffic flow on long
  crowded roads.
\newblock {\em Proceedings of the Royal Society. London. Series A.
  Mathematical, Physical and Engineering Sciences}, 229(1178):\ 317--345, 1955.

\bibitem{matthes}
D.~Matthes and H.~Osberger.
\newblock Convergence of a variational {L}agrangian scheme for a nonlinear
  drift diffusion equation.
\newblock {\em ESAIM: Mathematical Modelling and Numerical Analysis}, 48(3):\
  697--726, 2014.

\bibitem{Richards}
P.I. Richards.
\newblock Shock waves on the highway.
\newblock {\em Operations Research}, 4(1):\ 42--51, 1956.

\bibitem{rosini}
M.D. Rosini.
\newblock {\em {Macroscopic models for vehicular flows and crowd dynamics:
  theory and applications}}.
\newblock Understanding Complex Systems. Springer, Heidelberg, 2013.

\bibitem{RosSav}
R.~Rossi and G.~Savaré.
\newblock Tightness, integral equicontinuity and compactness for evolution
  problems in banach spaces.
\newblock {\em Annali della Scuola Normale Superiore di Pisa. Classe di
  Scienze}, 5(2):\ 395--431, 2003.

\bibitem{russo}
G.~Russo.
\newblock Deterministic diffusion of particles.
\newblock {\em Communications on Pure and Applied Mathematics}, 43(6):\
  697--733, 1990.

\bibitem{vara2}
D.~W. Stroock and S.~R.~S. Varadhan.
\newblock {\em Multidimensional diffusion processes}, volume 233 of {\em
  Grundlehren der Mathematischen Wissenschaften [Fundamental Principles of
  Mathematical Sciences]}.
\newblock Springer-Verlag, Berlin-New York, 1979.

\bibitem{toscani_kaniadakis}
G.~Toscani.
\newblock Finite time blow up in {K}aniadakis-{Q}uarati model of
  {B}ose-{E}instein particles.
\newblock {\em Communications in Partial Differential Equations}, 37(1):\
  77--87, 2012.

\bibitem{villani}
C.~Villani.
\newblock {\em Optimal transport. Old and new}, volume 338 of {\em Grundlehren
  der Mathematischen Wissenschaften [Fundamental Principles of Mathematical
  Sciences]}.
\newblock Springer-Verlag, Berlin, 2009.

\end{thebibliography}
\bibliographystyle{plain}

\end{document}